\documentclass[12pt]{article}
\usepackage[T1]{fontenc}
\usepackage[utf8]{inputenc}
\usepackage{lmodern}
\usepackage[margin=1in]{geometry}
\usepackage{amsmath,amssymb,amsthm,mathtools}
\usepackage{microtype}
\usepackage{enumitem}
\usepackage{booktabs,longtable}
\usepackage[hidelinks]{hyperref}
\numberwithin{equation}{section}
\newtheorem{theorem}{Theorem}[section]
\newtheorem{maintheorem}{Theorem}

\newtheorem{lemma}[theorem]{Lemma}
\newtheorem{proposition}[theorem]{Proposition}
\newtheorem{corollary}[theorem]{Corollary}
\theoremstyle{definition}
\newtheorem{definition}[theorem]{Definition}
\theoremstyle{remark}
\newtheorem{remark}[theorem]{Remark}
\DeclareMathOperator{\Cov}{Cov}
\DeclareMathOperator{\Var}{Var}
\DeclareMathOperator{\supp}{supp}
\newcommand{\E}{\mathbb E}
\newcommand{\R}{\mathbb R}

\newcommand{\weakstar}{\overset{*}{\rightharpoonup}}
\newcommand{\Cweak}{C_{\mathrm{weak}}}
\allowdisplaybreaks[1]
\hypersetup{pdftitle={Endpoint Compactness and Finite-Energy Weak Solutions to 2D  Compressible Navier-Stokes},pdfauthor={Cheng Yu}}

\title{Endpoint Compactness and Finite-Energy Weak Solutions for the 2D Isothermal Compressible Navier--Stokes}
\author{Cheng Yu\\\small Department of Mathematics, University of Florida}
\date{\today}
\begin{document}
\maketitle
\begin{abstract}

We prove global existence of finite-energy renormalized weak solutions
to the 2D isothermal compressible Navier--Stokes equations in bounded $C^2$ domains with no-slip boundary conditions.
The initial data are arbitrary in the natural energy class, allowing
vacuum and unbounded density. This resolves the isothermal endpoint
$\gamma=1$ of the classical 2D finite-energy renormalized existence
theory for power-law pressures. 
We develop a new endpoint compactness argument based on a
truncation-independent integrable bound for centered density
oscillations. This argument yields strong $L^1$ compactness of the
approximating densities at the natural energy level.

\end{abstract}

\section{Introduction}\label{sec:intro}

Global existence of finite-energy renormalized weak solutions at the 2D isothermal endpoint has remained a long-standing open problem for general natural-energy data.
In this paper, we construct global finite-energy renormalized
weak solutions for large initial data,
including vacuum and unbounded density, to the following
system:
\begin{align}
 \partial_t\rho+\operatorname{div}(\rho u)&=0,\label{eq:mass}\\
 \partial_t(\rho u)+\operatorname{div}(\rho u\otimes u)+\nabla\rho
 &=\operatorname{div}\mathbb S(u),\label{eq:momentum}\\
 \mathbb S(u)&=\mu(\nabla u+\nabla u^{\mathsf T})
                  +\lambda\operatorname{div}u\,I,\label{eq:stress}
\end{align}
in a bounded domain $\Omega\subset\mathbb R^2$ with $C^2$ boundary.
Here $\rho$ is the density and $u$ is the velocity. We impose the
no-slip boundary condition and assume that the constant viscosity
coefficients satisfy
\begin{equation}\label{eq:viscosity}
 u|_{\partial\Omega}=0,\qquad \mu>0,\qquad \mu+\lambda\ge0.
\end{equation}
The isothermal pressure law is $p(\rho)=\rho$, with its coefficient
normalized to one. The corresponding nonnegative internal energy is
\begin{equation}\label{eq:H}
 H(r)=r\log r-r+1\quad (r>0),\qquad H(0)=1.
\end{equation}

The initial density and momentum are prescribed by
\begin{equation}\label{intro:initial-data}
 \rho(0,\cdot)=\rho_0,\qquad
 (\rho u)(0,\cdot)=m_0
 \quad\text{in }\Omega,
\end{equation}
and satisfy only the natural finite-energy assumptions
\begin{equation}\label{eq:data}
 \rho_0\ge0,\quad H(\rho_0)\in L^1(\Omega),\quad
 m_0=0\ \text{a.e. on }\{\rho_0=0\},\quad
 |m_0|^2/\rho_0\in L^1(\Omega).
\end{equation}
In all kinetic quotients, $|z|^2/r$ is zero at $(r,z)=(0,0)$ and
is $+\infty$ when $r=0$ and $z\ne0$. We write
$E_0=\int_\Omega[|m_0|^2/(2\rho_0)+H(\rho_0)]\,dx$ for the
initial energy.

Our main result is that every pair of initial data satisfying
\eqref{eq:data} gives rise to a global  finite-energy renormalized
weak solution of \eqref{eq:mass}--\eqref{eq:stress}.
The result allows vacuum and unbounded initial density, with
no symmetry assumption. We also prove a sequential stability
theorem for the artificial-pressure approximation used in the
construction. Along a subsequence, the approximating densities
converge strongly in $L^1((0,T)\times\Omega)$ for every finite $T$.
The proof relies on an endpoint compactness argument that
recovers renormalized continuity from the distributional limit
before strong density convergence is known.

The main obstruction is the removal of the source in the truncated
continuity equation: the supercritical interpolation exponent
$(\gamma-1)/\gamma$ vanishes at $\gamma=1$. We establish effective flux
for each fixed density truncation while retaining this source.
On sets with bounded truncation mean, the resulting variance provides
a single truncation-independent integrable majorant for the squared
source coefficient. Dominated convergence then removes the renormalized
source, recovering renormalized continuity before the truncated-entropy
comparison yields strong density convergence.

For power pressures $p(r)=r^\gamma$, the finite-energy theory developed
by Lions \cite{Lions}, Feireisl \cite{Feireisl} and Feireisl--Novotn\'y--Petzeltov\'a
\cite{FNP} gives global weak solutions
for $\gamma>1$ in 2D; see \cite[p.~391]{FNP}. 
In his 1998 monograph, Lions \cite[Remark~7.8, p.~181]{Lions} identified the relaxation of the pressure-exponent restriction as a fundamental problem in the mathematical study of compressible Navier--Stokes equations. We resolve the 2D isothermal endpoint of this finite-energy
renormalized existence problem. Together with the classical results for $\gamma>1$, this closes the finite-energy existence range for power pressures
$p(\rho)=\rho^\gamma$ at the isothermal endpoint $\gamma=1$ in 2D.

The effective viscous flux plays a central role in this theory.
Its structural importance was recognized in the works of Hoff \cite{Hoff1995ARMA}, Serre \cite{Serre1991}
and Lions \cite{Lions}. Its weak-continuity
mechanism couples the momentum equation to renormalized continuity
and supplies the density oscillation control needed in the
Lions \cite{Lions}, Feireisl \cite{Feireisl} and Feireisl--Novotn\'y--Petzeltov\'a \cite{FNP}  compactness argument. Bresch and Jabin \cite{BJ} develop quantitative
density compactness estimates for more general pressure laws
and viscous stresses; Chaudhuri, Mucha and Zatorska \cite{CMZ}
provide a related weak solution construction.

At $\gamma=1$, however, the natural energy controls only
$\rho\log(1+\rho)$ and hence uniform integrability of the density;
it provides no $L^\infty_tL^{1+\varepsilon}_x$ bound for
$\varepsilon>0$. Although the linear pressure passes directly to
the limit under weak convergence, the inertial terms and the
nonlinear functions of density required for renormalized continuity
and entropy comparison still have to be identified.

For suitable density-dependent viscosities,
the Bresch--Desjardins entropy \cite{BD} gives additional spatial
control of the density. Mellet and Vasseur \cite{MV} combined this
structure with additional kinetic integrability to obtain weak
stability. Vasseur and Yu recovered a Mellet--Vasseur type estimate at the
weak-solution level and established global existence for large data
allowing vacuum, for $\gamma>1$ in 2D and
$1<\gamma<3$ in 3D.
Li and Xin
\cite{LX} obtained independent existence results. Later  Bresch,
Vasseur and Yu \cite{BVY} treated a broad class of nonlinear
density-dependent viscosities satisfying the BD relation.
These theories impose energy and entropy conditions adapted to the
viscosity law. The present problem
has constant viscosities and only the natural finite-energy data
in \eqref{eq:data}.

For the 2D isothermal system, Plotnikov and Weigant
\cite{PW} overcome the difficulty of identifying inertia through
a Radon-transform estimate. Their argument yields local density
integrability $\rho\in L^{1+\alpha}_{\mathrm{loc}}$ for
$0<\alpha<1/6$, together with the negative-Sobolev compactness
needed to identify momentum and convection. Their existence
theorem assumes
\[
 \rho_0\in L^\infty(\Omega),\qquad \rho_0\ge c>0,\qquad
 u_0\in W^{1,2}_0(\Omega;\R^2)\cap L^\infty(\Omega;\R^2).
\]
The solution concept in \cite[Definition~1.2]{PW} is distributional;
the limiting argument in Section~7 of that work does not establish renormalized
continuity or strong density convergence. Weigant and Plotnikov
\cite{WP2017} extend density and kinetic integrability estimates
to the boundary and explicitly identify the existence of
renormalized solutions as an open problem
\cite[p.~1116]{WP2017}. The handbook treatment
\cite[Condition~1 and Theorem~2, pp.~1603--1605]{PWHandbook}
retains the same initial-data assumptions and proves weak existence
with local density gain for $0<\alpha<1/8$; renormalized continuity
of the limit and strong density convergence are not conclusions
of that theorem.\footnote{The renormalization lemma
\cite[Lemma~3, p.~1620]{PWHandbook} is applied to artificial-pressure
approximations with densities in $L^\infty_tL^6_x$. The
2D argument in
\cite[Section~5, pp.~1624--1637]{PWHandbook} derives local density
and kinetic integrability and refers to \cite{PW} for further details.}

We use the Radon-transform and local-density estimates of \cite{PW} to identify momentum and convection. This yields a distributional limit. Their required dependence on the natural energy quantities
is verified in Subsection~\ref{appendix:PW} of
Appendix~\ref{appendix:approx}. The endpoint argument developed
below then establishes fixed-truncation effective flux, renormalized
continuity, and strong $L^1$ density convergence. Together with
the energy-admissible approximation scheme, this allows the
construction from arbitrary data in \eqref{eq:data}, including
vacuum and unbounded density.

The remaining obstacle is the removal of the source in the
truncated continuity equation. For the standard smooth concave density
truncations $T_k(r)=kT(r/k)$, which we also call caps, set
\[
 a_k(r)=rT_k'(r)-T_k(r).
\]

The classical source-removal arguments in
\cite[Proposition~7.1, (7.6)]{Feireisl}
and \cite[Lemma~4.4, (4.25)]{FNP}
interpolate a small $L^1$ tail of $a_k$
with an $L^{\gamma+1}$ oscillation bound.
In the estimate for the squared $L^2$ norm, the exponent of the small
$L^1$ factor is $(\gamma-1)/\gamma$. At $\gamma=1$, this exponent
vanishes, so the classical interpolation step no longer yields the
residual smallness needed for source removal.
Thus the classical source-removal step does not close even after
inertia and fixed-cap effective flux have been identified.

We now state the main results. Theorem~\ref{thm:main} is our main
existence result; Theorem~\ref{thm:stability} gives the endpoint
compactness used in its construction. We first specify the notion
of weak solution used in the paper. For $T>0$, write
$Q_T=(0,T)\times\Omega$ and $Q_D=(0,T)\times D$ for
$D\Subset\Omega$. We use $L^p_tL^q_x(D)$ for the corresponding
mixed spaces; $W^{-1,q}(B)$ is the dual of $W^{1,q'}_0(B)$,
where $q'=q/(q-1)$. The notation $L\log L$ refers to
$r\log(1+r)$. The initial data satisfy \eqref{eq:data}, and
their energy is $E_0$. A smooth test on
$[0,T)\times\overline\Omega$ means the restriction of a smooth
function on a neighborhood, with compact support away from
the terminal time.
\begin{definition}[Finite-energy renormalized weak solution]
\label{def:renormalized-solution}
Let $T>0$, 
and assume \eqref{eq:viscosity} and \eqref{eq:data}.
A pair $(\rho,u)$, with $\rho\ge0$ and $m=\rho u$, is a
\emph{finite-energy renormalized weak solution} of the
isothermal system \eqref{eq:mass}--\eqref{eq:stress} on $Q_T$ with
initial data $(\rho_0,m_0)$ if the following conditions hold.
\begin{enumerate}[label=\textup{(\roman*)},leftmargin=*,itemsep=\smallskipamount]
\item \emph{Regularity and initial traces.} We have
\begin{align}
 \rho&\in L^\infty(0,T;L\log L(\Omega))
       \cap\Cweak([0,T];L^1(\Omega)),\label{eq:rho-class}\\
 u&\in L^2(0,T;W^{1,2}_0(\Omega;\R^2)),\qquad
 \sqrt\rho\,u\in L^\infty(0,T;L^2(\Omega;\R^2)),\label{eq:u-class}\\
 m&\in\Cweak([0,T];L^1(\Omega;\R^2)),\qquad
 \rho(0)=\rho_0,\quad m(0)=m_0.\label{eq:m-class}
\end{align}

\item \emph{Weak continuity and momentum equations.} For every smooth
scalar test $\psi$ on $[0,T)\times\overline\Omega$,
\begin{equation}\label{weak:mass}
 \int_0^T\!\int_\Omega(\rho\partial_t\psi+m\cdot\nabla\psi)
 +\int_\Omega\rho_0\psi(0)=0.
\end{equation}
For every smooth vector test $\varphi$ on the same set with zero
boundary trace,
\begin{equation}\label{weak:mom}
 \int_0^T\!\int_\Omega
 [m\cdot\partial_t\varphi+(\rho u\otimes u):\nabla\varphi
       +\rho\operatorname{div}\varphi-\mathbb S(u):\nabla\varphi]
 +\int_\Omega m_0\cdot\varphi(0)=0.
\end{equation}

\item \emph{Renormalized continuity.} For every
$b\in C^1([0,\infty))$ whose derivative is compactly supported,
\begin{equation}\label{eq:renormalized}
 \partial_t b(\rho)+\operatorname{div}(b(\rho)u)
 +(\rho b'(\rho)-b(\rho))\operatorname{div}u=0
\end{equation}
with initial value $b(\rho_0)$. Precisely, for every scalar test $\psi$
as in \textup{(ii)},
\begin{equation}\label{weak:renorm}
 \int_0^T\!\int_\Omega
 [b(\rho)\partial_t\psi+b(\rho)u\cdot\nabla\psi
  -(\rho b'(\rho)-b(\rho))\operatorname{div}u\,\psi]
 +\int_\Omega b(\rho_0)\psi(0)=0.
\end{equation}

\item \emph{Initial-time energy inequality.} For almost every
$t\in(0,T)$,
\begin{equation}\label{eq:energy}
 \int_\Omega\left[\frac12\rho|u|^2+H(\rho)\right](t)\,dx
 +\int_0^t\!\int_\Omega\mathbb S(u):\nabla u\,dx\,ds
 \le E_0.
\end{equation}
\end{enumerate}
A solution on $[0,\infty)$ is \emph{global} if these conditions hold
on every finite time interval with the same initial data.
\end{definition}

The continuity formulations may equivalently be read after extending
$\rho,u$ by zero outside $\Omega$ and $b(\rho)$ by $b(0)$.
The velocity in the definition is the Sobolev function in
\eqref{eq:u-class}, including on the vacuum set.

We now state our two main results.
Theorem~A establishes endpoint sequential stability and compactness,
while Theorem~B gives global finite-energy existence for arbitrary
natural-energy data.

For fixed $\delta>0$ and $\beta\ge6$, we consider weak solutions
$(r,v)$ with pressure $p_\delta(r)=r+\delta r^\beta$ and initial density
$r_0$. The weak equations and initial terms are understood as in
\eqref{weak:mass}--\eqref{weak:mom}, with $p_\delta(r)$ in the momentum
equation. On each finite interval $(0,T)$ on which the solution is
defined, we require
\[
 \begin{gathered}
 r\in L^\infty(0,T;L^\beta(\Omega))\cap L^{\beta+1}(Q_T)
      \cap\Cweak([0,T];L^\beta(\Omega)),\qquad r\ge0,\\
 v\in L^2(0,T;W^{1,2}_0(\Omega;\mathbb R^2)),\qquad
 \sqrt r\,v\in L^\infty(0,T;L^2(\Omega;\mathbb R^2)),\\
 rv\in\Cweak([0,T];L^{2\beta/(\beta+1)}(\Omega;\mathbb R^2)).
 \end{gathered}
\]
The renormalized continuity identity \eqref{eq:renormalized}, with
$(\rho,u)$ replaced by $(r,v)$ and initial term $b(r_0)$, is required
for every $b\in C^1([0,\infty))$ whose derivative is compactly supported.
The same identity is required for every $b\in C^2([0,\infty))$ satisfying
\[
 |b(r)|+|rb'(r)|+|r^2b''(r)|\le C_b(1+r^2).
\]
These fixed-parameter bounds need not be uniform in $\delta$.
The construction for regularized initial data is given in
Proposition~\ref{app:fixed}.

\begin{maintheorem}[Endpoint sequential stability]\label{thm:stability}
Assume \eqref{eq:viscosity}. Let $\beta\ge6$ be fixed and let
$\delta_n\downarrow0$. Let $(\rho_n,u_n)$ be renormalized weak solutions
with no-slip velocity and pressure $\rho_n+\delta_n\rho_n^\beta$,
satisfying the fixed-parameter regularity and renormalization properties
stated above. Write $m_n=\rho_nu_n$.

Suppose the initial density and momentum converge in $L^1$ to
$(\rho_0,m_0)$ satisfying \eqref{eq:data}, their natural energy integrals
converge to $E_0$, and $\delta_n\int_\Omega\rho_{0,n}^\beta\to0$.
Assume that, on each finite interval $(0,T)$ on which the sequence is
defined, the following energy inequality holds for almost every
$t\in(0,T)$, with $E_{0,n}$ denoting the corresponding initial energy:
\begin{equation}\label{exist:approxenergy}
 E_n(t)+\int_0^t\!\int_\Omega\mathbb S(u_n):\nabla u_n\le E_{0,n},
 \quad E_n(t)=\int_\Omega\left[
 \frac12\rho_n|u_n|^2+H(\rho_n)+\frac{\delta_n}{\beta-1}\rho_n^\beta
 \right](t),
\end{equation}

On each such interval, assume in addition that
\[
 \underset{0<t<T}{\operatorname{ess\,sup}}
 \int_\Omega\left(\rho_n|u_n|^2+\rho_n\log(1+\rho_n)
                    +\delta_n\rho_n^\beta\right)dx
 +\int_{Q_T}|\nabla u_n|^2\,dx\,dt\le C
\]
and
\[
 \int_{Q_T}\zeta^2\bigl(\rho_n+\delta_n\rho_n^\beta\bigr)
                  \rho_n^{1/8}\,dx\,dt\le C_{\zeta,T}
 \qquad\text{for every }0\le\zeta\in C_c^\infty(\Omega),
\]
with constants independent of $n$.
Then, on every finite time interval $(0,T)$ on which the sequence is defined, a
subsequence satisfies
\begin{align*}
 \rho_n&\longrightarrow\rho&&\text{strongly in }L^1(Q_T),\\
 u_n&\rightharpoonup u&&\text{weakly in }L^2_tW^{1,2}_0,\\
 m_n&\rightharpoonup\rho u&&\text{weakly in }L^1(Q_T),\\
 \chi m_n&\longrightarrow\chi\rho u
       &&\text{strongly in }L^2_tW^{-1,192/95}_x
               \quad(\chi\in C_c^\infty(\Omega)),\\
 \rho_nu_n\otimes u_n&\rightharpoonup\rho u\otimes u
       &&\text{locally in }L^{6/5}_tL^{192/191}_x.
\end{align*}
The limit is a finite-energy renormalized weak solution in the sense of
Definition~\ref{def:renormalized-solution} on this interval. Artificial
pressure converges to zero in $L^1(Q_D)$ for every $D\Subset\Omega$.
\end{maintheorem}

Theorem~A supplies the compactness needed for the construction of
solutions from arbitrary natural-energy data. This yields our second
main result.

\begin{maintheorem}[Global finite-energy existence]\label{thm:main}
Assume
\eqref{eq:viscosity}. Let the initial data satisfy \eqref{eq:data}.
Then there exists a global finite-energy renormalized weak solution
$(\rho,u)$ of \eqref{eq:mass}--\eqref{eq:stress} in the sense of
Definition~\ref{def:renormalized-solution}.
The solution is obtained as the limit of artificial-pressure approximants
whose densities converge strongly in $L^1(Q_T)$ for every finite $T$.
\end{maintheorem}

The key point is
to use the available quadratic oscillation control as an integrable
majorant for the cap source.
Let $\eta_{t,x}$ be the density Young measure and let $X$ denote
its coordinate variable. For an integrable function $f$, write
\[
 \E f(X):=\int_{[0,\infty)} f(r)\,\mathrm d\eta_{t,x}(r),
\]
suppressing the dependence on $(t,x)$.
The symbols $\Cov$ and $\Var$ denote covariance and variance
with respect to the same measure. Thus $\E X=\rho$, and we set
\[
 \begin{aligned}
 t_k&:=\E T_k(X),\\
 C_k&:=\Cov(X,T_k(X))
       =\E[X T_k(X)]-\rho\,t_k,\\
 V_k&:=\Var(T_k(X))
       =\E|T_k(X)-t_k|^2.
 \end{aligned}
\]
The caps are increasing, concave and one-Lipschitz, agree with
the identity on $[0,k]$, and equal $2k$ on $[3k,\infty)$.
The fixed-cap effective-flux identity reads
\begin{equation}\label{intro:flux}
 C_k=\nu(h_k-t_k\operatorname{div}u),\qquad \nu=2\mu+\lambda>0,
\end{equation}
where $h_k$ is the weak limit of
$T_k(\rho_n)\operatorname{div}u_n$.
We prove this identity from the distributional limit equations,
using the 2D compactness estimates of \cite{PW}
and retaining the cap source throughout the passage.

The monotone pressure-product inequalities of Lions \cite{Lions}
and the cap-oscillation estimates of Feireisl
\cite[Proposition~6.1]{Feireisl} and Feireisl, Novotn\'y and
Petzeltov\'a \cite[Lemma~4.3]{FNP} provide the underlying
quadratic information. At linear pressure, we express it in
the centered form
\begin{equation}\label{intro:variance}
 0\le V_k\le C_k,\qquad C_k\uparrow V:=\E|X-\rho|^2.
\end{equation}
The oscillation estimate obtained from \eqref{intro:flux} yields
\begin{equation}\label{intro:bound}
 \int_{Q_T}V\le\nu^2D_*,\qquad
 D_*:=\sup_n\int_{Q_T}|\operatorname{div}u_n|^2,
 \qquad Q_T=(0,T)\times\Omega.
\end{equation}
This controls the variance of the Young measure rather than the
square of its mean: a Dirac measure has zero variance even when
the density is not in $L^2(Q_T)$.

The new ingredient is to use this integrable centered
variance as a cap-independent majorant for the truncation source.
More precisely, for every fixed $M<\infty$ and $k>2M$,
\begin{equation}\label{intro:majorant}
 \mathbf1_{\{t_k\le M\}}\E|a_k(X)|^2\le4V,
 \qquad V\in L^1(Q_T).
\end{equation}
The source coefficient vanishes on $[0,k]$. On the remaining
tail, a bounded cap mean gives
$|a_k(X)|\le T_k(X)\le2|T_k(X)-t_k|$, which converts
the centered oscillation bound into \eqref{intro:majorant}.
Finite variance and finite mean also give a finite second moment
in the Young probability space at almost every $(t,x)$, so
$\E|a_k(X)|^2\to0$ pointwise. Dominated convergence with the
single majorant $4V$ therefore gives
\begin{equation}\label{intro:residual}
 \int_{\{t_k\le M\}}\E|a_k(X)|^2\longrightarrow0.
\end{equation}
This estimate supplies the smallness lost in the classical
interpolation argument at $\gamma=1$. The integrable variance
majorant is sufficient even when the mean density is not
square integrable.

To use this estimate, we first renormalize the bounded cap-mean
equation
\[
 \partial_t t_k+\operatorname{div}(t_ku)=-S_k,
 \qquad S_k=\mathop{\mathrm{w}\!\!-\!\!\lim}_{n\to\infty}
               a_k(\rho_n)\operatorname{div}u_n.
\]
The transport commutator applies to $t_k$ before strong density
convergence is available. If $b'$ has compact support, the
renormalized source $b'(t_k)S_k$ is supported on a bounded
cap-mean set. Cauchy--Schwarz and \eqref{intro:residual} remove
this source as $k\to\infty$, giving renormalized continuity
for $\rho$. The classical truncated-entropy comparison then
eliminates the remaining density oscillation and yields strong
$L^1$ convergence. The initial entropy trace needed in this
comparison is proved separately by compactness in time of the
spatial entropy integrals. Thus renormalization is recovered
before the entropy defect collapses, and strong density
convergence follows only at the end of the argument.

The rest of the paper is organized as follows.
Section~\ref{sec:approx} constructs energy-admissible
artificial-pressure approximations and collects the uniform
estimates. Section~\ref{sec:momentum} identifies momentum
and convection before strong density convergence is known,
and Section~\ref{sec:flux} establishes fixed-cap effective flux.
Section~\ref{sec:variance} derives the integrable majorant for
the cap source. Section~\ref{sec:closure} uses it to recover
renormalized continuity and then closes the truncated-entropy
argument, yielding strong density convergence. The density,
momentum, cap, and truncated-entropy traces are established
at the stages where they are needed.
Section~\ref{sec:existence} completes the energy and boundary
passages and proves the two main theorems. The approximation
details and the natural-energy dependence of the estimates
from \cite{PW} are given in Appendix~\ref{appendix:approx}.

Throughout the paper, all expectations refer to the single
density Young measure $\eta_{t,x}$. The symbols $t_k,P_k,h_k,S_k$
denote the cap, pressure--cap, cap--divergence and cap-source weak
limits. Constants may depend on the fixed time interval and localization;
cap dependence is permitted only while the cap is fixed.

\section{Approximation and uniform estimates}\label{sec:approx}

We collect the approximation results needed for the endpoint argument.
The initial-data regularization and the fixed high-pressure construction
follow the finite-energy framework of \cite{FNP,Lions}; the local
estimates are those of \cite{PW}. We state the conclusions used below
and give the details of the adaptations to the isothermal setting
and natural-energy data in Appendix~\ref{appendix:approx}.

\subsection{Approximation of the initial data}

\begin{lemma}[Energy-consistent initial approximation]\label{app:data}
Suppose that $(\rho_0,m_0)$ satisfies the initial-data assumptions of the
main theorem. There are functions $\rho_{0,n}$, smooth on a neighborhood
of $\overline\Omega$, and $u_{0,n}\in C_c^\infty(\Omega;\mathbb R^2)$,
such that $0<c_n\leq\rho_{0,n}\leq C_n<\infty$ and, with
$m_{0,n}=\rho_{0,n}u_{0,n}$,
\begin{align}
 \rho_{0,n}&\longrightarrow\rho_0,
 &H(\rho_{0,n})&\longrightarrow H(\rho_0),
 &m_{0,n}&\longrightarrow m_0 &&\text{in }L^1(\Omega),\label{app:dataL1}\\
 \sqrt{\rho_{0,n}}u_{0,n}&\longrightarrow
 \frac{m_0}{\sqrt{\rho_0}} &&&&&\text{in }L^2(\Omega;\mathbb R^2).
 \label{app:weighteddata}
\end{align}
The last quotient is zero on $\{\rho_0=0\}$.
\end{lemma}

\begin{proof}
The weighted kinetic approximation and the logarithmic-energy
regularization are given in Appendix~\ref{appendix:initial}.
\end{proof}

Fix $\beta\ge6$. With $A_j=\int_\Omega\rho_{0,j}^\beta\,dx$, set
\begin{equation}\label{app:delta}
 \delta_n=\frac{1}{n(1+\sum_{j=1}^n A_j)},\qquad
 p_n(r)=r+\delta_n r^\beta.
\end{equation}
Then $\delta_n\downarrow0$ and $\delta_n A_n\le n^{-1}$, so
\begin{equation}\label{app:initialenergy}
 E_{0,n}:=\int_\Omega
 \left(\frac12\rho_{0,n}|u_{0,n}|^2+H(\rho_{0,n})
       +\frac{\delta_n}{\beta-1}\rho_{0,n}^\beta\right)dx
 \longrightarrow
 E_0:=\int_\Omega\left(\frac{|m_0|^2}{2\rho_0}+H(\rho_0)\right)dx.
\end{equation}

\subsection{An energy-admissible artificial-pressure solution}

For fixed $\delta>0$, write
\[
 p_\delta(r)=r+\delta r^\beta,\qquad
 H_\delta(r)=H(r)+\frac{\delta}{\beta-1}r^\beta.
\]
The coercivity identity, valid also when $\mu+\lambda=0$, is
\begin{equation}\label{app:coercivity}
 \int_\Omega\mathbb S(w):\nabla w
 =\mu\|\nabla w\|_2^2+(\mu+\lambda)\|\operatorname{div}w\|_2^2
 \ge\mu\|\nabla w\|_2^2,
 \qquad w\in W^{1,2}_0(\Omega;\mathbb R^2).
\end{equation}

\begin{proposition}[Fixed artificial pressure]\label{app:fixed}
Let $\delta>0$, $\beta\geq6$, $0<c\leq r_0\leq C$, with $r_0$ smooth
on a neighborhood of $\overline\Omega$, and let
$v_0\in C_c^\infty(\Omega;\mathbb R^2)$. There is a global solution
$(r,v)$ of
\begin{equation}\label{app:fixedsystem}
 \partial_t r+\operatorname{div}(rv)=0,\qquad
 \partial_t(rv)+\operatorname{div}(rv\otimes v)+\nabla p_\delta(r)
       =\operatorname{div}\mathbb S(v),\qquad v|_{\partial\Omega}=0,
\end{equation}
with the prescribed density and momentum initial terms $r_0$ and $r_0v_0$.
For every $T<\infty$,
\begin{gather*}
 r\in L^\infty(0,T;L^\beta(\Omega))\cap L^{\beta+1}(Q_T)
          \cap C_{\mathrm{weak}}([0,T];L^\beta(\Omega)),\qquad r\geq0,\\
 v\in L^2(0,T;W^{1,2}_0(\Omega;\mathbb R^2)),\qquad
 \sqrt r\,v\in L^\infty(0,T;L^2(\Omega;\mathbb R^2)),\\
 rv\in C_{\mathrm{weak}}([0,T];L^{2\beta/(\beta+1)}(\Omega;\mathbb R^2)).
\end{gather*}
For every $b\in C^1([0,\infty))$ whose derivative is compactly
supported, continuity is renormalized:
\begin{equation}\label{app:fixed-renormalized}
 \partial_t b(r)+\operatorname{div}(b(r)v)
 +(rb'(r)-b(r))\operatorname{div}v=0,
\end{equation}
with initial term $b(r_0)$. The same identity also holds for
\begin{equation}\label{app:PWclass}
 b\in C^2([0,\infty)),\qquad
 |b(r)|+|rb'(r)|+|r^2b''(r)|\leq C_b(1+r^2).
\end{equation}
For almost every $t>0$,
\begin{equation}\label{app:fixedenergy}
 E_\delta(t)+\int_0^t\!\int_\Omega\mathbb S(v):\nabla v\,dx\,ds
 \leq E_\delta(0),\qquad
 E_\delta(t)=\int_\Omega\left(\frac12r|v|^2+H_\delta(r)\right)(t)dx.
\end{equation}
\end{proposition}

\begin{proof}
We adapt the fixed high-pressure scheme of \cite[Sections~2--3]{FNP}
to the planar domain and logarithmic physical potential. The monotone
pressure law is also covered, on smooth bounded domains, by
\cite[Theorem~7.2, p.~180, Section~7.5, p.~205, and Example~7.1, p.~207]{Lions}.
Appendix~\ref{appendix:fixed} gives the $C^2$ parabolic input and the
passages $N\to\infty$, then $\varepsilon\downarrow0$ at fixed $\delta$,
retaining the prescribed initial energy and the stated renormalization
class. Constants in this construction may depend on $\delta$.
\end{proof}

\subsection{The Radon estimate and its energy dependence}

Apply Proposition~\ref{app:fixed} to \eqref{app:delta} and choose one
global solution $(\rho_n,u_n)$ for each $n$. We retain this sequence
throughout the paper. By \eqref{app:initialenergy},
\eqref{app:fixedenergy} and \eqref{app:coercivity},
\begin{equation}\label{app:uniformenergy}
 \underset{0<t<T}{\operatorname{ess\,sup}}
 \int_\Omega\left(\rho_n|u_n|^2+\rho_n\log(1+\rho_n)
                    +\delta_n\rho_n^\beta\right)dx
 +\int_{Q_T}|\nabla u_n|^2\,dx\,dt\le C.
\end{equation}
Mass, $\|m_{0,n}\|_1$ and $\|p_n(\rho_n)\|_{L^\infty_tL^1_x}$
are therefore uniformly bounded.

\begin{lemma}[Localized estimates of Plotnikov--Weigant]\label{app:PW}
Let $\zeta\in C_c^\infty(\Omega)$ be nonnegative and
$0<\alpha<1/6$. The selected solutions satisfy
\begin{align}
 \sup_{\omega\in\mathbb S^1}\int_0^T\!\int_{\mathbb R}
 \left(\int_{\{\omega\cdot x=\tau\}}\zeta(x)\rho_n(t,x)\,d\ell\right)^2
 d\tau\,dt&\leq C_{\zeta,T},\label{app:Radon}\\
 \|\zeta\rho_n\|_{L^2(0,T;H^{-1/2}(\mathbb R^2))}
 &\leq C_{\zeta,T},\label{app:halfderivative}\\
 \int_{Q_T}\zeta^2\bigl(\rho_n+\delta_n\rho_n^\beta\bigr)
                         \rho_n^\alpha\,dx\,dt
 &\leq C_{\zeta,\alpha,T}.\label{app:localgain}
\end{align}
The zero extension is used in \eqref{app:Radon}--\eqref{app:halfderivative}.
The constants depend only on the uniform initial augmented energy,
mass, $T$, $\Omega$, $\mu$, $\lambda$, the fixed $\beta$, and the displayed
localization and exponent. They require no common bound on
$\|\rho_{0,n}\|_\infty$, $(\inf\rho_{0,n})^{-1}$, or
$\|u_{0,n}\|_{W^{1,2}}$.
\end{lemma}

\begin{proof}
Apply the proofs of \cite[Theorem~4.1, p.~9, Proposition~5.1,
p.~17, and Theorem~6.1, p.~22]{PW} to these fixed-pressure solutions.\footnote{Numbered
references to \cite{PW} in Section~2 and Appendix~\ref{appendix:PW}
follow the authors' manuscript,
\href{https://arxiv.org/abs/1407.2743v1}{arXiv:1407.2743v1}.}
Appendix~\ref{appendix:PW} traces their constants to mass, kinetic
energy, dissipation, total pressure and initial momentum. It also
records the finite potential exponents and removal of spatial
mollification at fixed $n$, establishing the stated uniformity.
\end{proof}

\section{Momentum compactness and the distributional limit}\label{sec:momentum}

Fix $T>0$ and write
\[
 m_n=\rho_nu_n,\qquad F_n=\rho_nu_n\otimes u_n,\qquad
 p_n=\rho_n+\delta_n\rho_n^\beta,\qquad S_n=\mathbb S(u_n).
\]
We use the uniform energy, the equations and the local density gain at
$\alpha=1/8$ to identify momentum and convection before effective flux.

\subsection{Weak compactness and finite mixed exponents}

The entropy and kinetic bounds give, for measurable $A\subset\Omega$
and $R\ge2$,
\begin{align}
 \sup_{n,t}\int_A\rho_n(t)\,dx
 &\le R|A|+\frac{C}{\log(1+R)},\label{mom:UI}\\
 \sup_{n,t}\int_A|m_n(t)|\,dx
 &\le C\left(\sup_{n,t}\int_A\rho_n(t)\,dx\right)^{1/2}.
 \label{mom:mUI}
\end{align}
Here and below the supremum in time is essential. Applying these
bounds to time slices also gives uniform integrability on $Q_T$.
After extraction,
\begin{equation}\label{mom:weak}
 \rho_n\rightharpoonup\rho,\quad m_n\rightharpoonup m
 \quad\text{in }L^1(Q_T),\qquad
 u_n\rightharpoonup u\quad\text{in }L^2_tW^{1,2}_0(\Omega).
\end{equation}
In particular, the weak physical pressure is already $\rho$.

\begin{lemma}[Finite local exponents]\label{mom:exponents}
For every $D\Subset\Omega$,
\begin{equation}\label{mom:mixed}
 \begin{aligned}
  \rho_n&\text{ is bounded in }L^4(0,T;L^{32/31}(D)),\\
  m_n&\text{ is bounded in }L^3(0,T;L^{96/95}(D)),\\
  F_n&\text{ is bounded in }L^{6/5}(0,T;L^{192/191}(D)).
 \end{aligned}
\end{equation}
\end{lemma}

\begin{proof}
Since $p_n\ge\rho_n$, \eqref{app:localgain} gives a uniform
$L^{9/8}(Q_D)$ bound. Interpolation with mass yields
\[
 \|\rho_n\|_{L^4_tL^{32/31}_x(D)}
 \le\|\rho_n\|_{L^\infty_tL^1_x(D)}^{23/32}
      \|\rho_n\|_{L^{9/8}(Q_D)}^{9/32}\le C_D.
\]
Using kinetic energy and the planar embedding
$W^{1,2}_0(\Omega)\hookrightarrow L^q(\Omega)$ for every finite $q$,
H\"older's inequality gives
\begin{align*}
 \|m_n\|_{L^3_tL^{96/95}_x(D)}
 &\le\|\rho_n\|_{L^4_tL^{32/31}_x(D)}^{2/3}
       \|\rho_n|u_n|^2\|_{L^\infty_tL^1_x(D)}^{1/3}
       \|u_n\|_{L^2_tL^{32}_x(D)}^{1/3}\le C_D,\\
 \|F_n\|_{L^{6/5}_tL^{192/191}_x(D)}
 &\le\|m_n\|_{L^3_tL^{96/95}_x(D)}
       \|u_n\|_{L^2_tL^{192}_x(D)}\le C_D.
\end{align*}
\end{proof}

\subsection{The compact momentum space}

\begin{proposition}[Local negative-Sobolev compactness]\label{mom:compact}
For each $\chi\in C_c^\infty(\Omega)$ and each smooth ball $B$ containing
its support, with the localized functions extended by zero, a subsequence
satisfies
\begin{align}
 \chi\rho_n&\longrightarrow\chi\rho
       &&\text{strongly in }L^2(0,T;H^{-1}(B)),\label{mom:rho-compact}\\
 \chi m_n&\longrightarrow\chi m
       &&\text{strongly in }L^2(0,T;W^{-1,192/95}(B)).\label{mom:exact}
\end{align}
In particular localized momentum converges strongly in $L^2_tH^{-1}_x$.
\end{proposition}

\begin{proof}
The planar Sobolev embedding and its adjoint give
\begin{equation}\label{mom:embeddings}
 W^{1,192/97}_0(B)\Subset L^{96}(B),\qquad
 L^{96/95}(B)\Subset W^{-1,192/95}(B)\hookrightarrow H^{-3}(B).
\end{equation}
The first embedding is strictly below the critical exponent $192$.
Componentwise, the momentum equation reads
\begin{equation}\label{mom:derivative}
 \partial_t(\chi m_{n,i})
 =\partial_j[\chi(S_{n,ij}-F_{n,ij})]
  -(\partial_j\chi)(S_{n,ij}-F_{n,ij})
  -\partial_i(\chi p_n)+p_n\partial_i\chi.
\end{equation}
Since $F_n,p_n$ are bounded in $L^\infty_tL^1_x$,
$S_n$ in $L^2(Q_T)$ and $H^3_0(B)\hookrightarrow C^1(\overline B)$,
this derivative is bounded in $L^2_tH^{-3}_x$. The same is true of
$\partial_t(\chi\rho_n)=-\operatorname{div}(\chi m_n)+m_n\cdot\nabla\chi$.
Simon compactness \cite{Simon}, \eqref{mom:mixed} and
\eqref{mom:embeddings} imply \eqref{mom:exact}. For density use
$L^{32/31}(B)\Subset H^{-1}(B)\hookrightarrow H^{-3}(B)$.
Finally $W^{-1,192/95}(B)\hookrightarrow H^{-1}(B)$.
\end{proof}

The momentum space in \eqref{mom:exact} is dual to the
$W^{1,192/97}$ commutator space used in Section~\ref{sec:flux}.

\subsection{Identification before effective flux}

\begin{proposition}[Momentum and convection]\label{mom:identification}
The weak limits satisfy $m=\rho u$ almost everywhere, and for every
$D\Subset\Omega$,
\begin{equation}\label{mom:convection}
 F_n\rightharpoonup\rho u\otimes u
       \quad\text{in }L^{6/5}(0,T;L^{192/191}(D)).
\end{equation}
\end{proposition}

\begin{proof}
The reflexive bounds in \eqref{mom:mixed} pass to $\rho$ and $m$.
Let $\varphi$ and $\Phi$ be smooth compactly supported vector and
matrix tests, and choose $\chi=1$ near their spatial supports, with
$\supp\chi\Subset B$. In time-integrated $H^{-1}(B)$--$H^1_0(B)$
duality,
\begin{equation}\label{convergence of mass in dis}
 \begin{aligned}
 \int_{Q_T}(\rho_nu_n-\rho u)\cdot\varphi
 &=\langle\chi(\rho_n-\rho),u_n\cdot\varphi\rangle
   +\langle\chi\rho,(u_n-u)\cdot\varphi\rangle,\\
 \int_{Q_T}(F_n-m\otimes u):\Phi
 &=\langle\chi(m_n-m),\Phi u_n\rangle
   +\langle\chi m,\Phi(u_n-u)\rangle.
 \end{aligned}
\end{equation}
The first terms vanish by Proposition~\ref{mom:compact}; the second
by weak Sobolev convergence of $u_n$. All pairings agree with ordinary
integrals: density--velocity products lie in $L^{4/3}_tL^1_x$, and
momentum--velocity products in $L^{6/5}_tL^1_x$, by
\eqref{mom:mixed} and finite planar Sobolev embeddings.
The first identity identifies $m=\rho u$ by comparison with
\eqref{mom:weak}. The second identifies the distributional limit of
$F_n$ as $m\otimes u=\rho u\otimes u$. Its reflexive mixed bound
then gives \eqref{mom:convection}.
\end{proof}

\begin{lemma}[Removal of artificial pressure]\label{mom:artificial}
For every $D\Subset\Omega$,
\begin{equation}\label{mom:AP}
 \delta_n\rho_n^\beta\longrightarrow0
       \quad\text{strongly in }L^1(Q_D).
\end{equation}
\end{lemma}
\begin{proof}
Using \eqref{app:localgain} with a cutoff equal to one on $D$,
\[
 \int_{Q_D}\delta_n\rho_n^\beta
 \le\delta_nR^\beta T|D|
    +R^{-1/8}\int_{Q_D}\delta_n\rho_n^{\beta+1/8}
 \le\delta_nR^\beta T|D|+C_{D,T}R^{-1/8}.
\]
Let $n\to\infty$, then $R\to\infty$.
\end{proof}
This conclusion holds on the full interval $(0,T)$ and remains valid
after multiplication by a fixed bounded density cap. It suffices for
all sequence-level momentum tests, which are spatially interior.

\subsection{Initial traces and the preliminary limit equations}

Global weak $L^1$ convergence passes the continuity equation for tests
on $\overline\Omega$, its initial term and conservation of mass.
The preceding results give the momentum equation for interior tests
with initial datum $m_0$. Convex lower semicontinuity of
$(r,z)\mapsto |z|^2/(2r)+H(r)$ against nonnegative time weights gives
the natural uniform energy bound for $(\rho,m)$; with $m=\rho u$ and
$u\in L^2_tH^1_0$, it also gives
$\rho u\otimes u\in L^\infty_tL^1_x$. The exact initial-time energy
inequality is passed in Section~\ref{sec:existence}.

Smooth spatial test pairings are continuous by the weak equations.
The inherited bounds \eqref{mom:UI}--\eqref{mom:mUI} give weakly compact,
uniformly integrable ranges and hence representatives at every time.
Uniformly bounded smooth approximations of any $L^\infty$ test in
measure have pairing errors uniformly small in time by these same
bounds; for momentum the approximations have interior support.
Consequently,
\begin{equation}\label{mom:traces}
 \rho,m\in\Cweak([0,T];L^1(\Omega)),\qquad
 \rho(0)=\rho_0,\quad m(0)=m_0,
\end{equation}
with the vector-valued interpretation for $m$. Nonlinear initial
traces are treated in Section~\ref{sec:closure}.

We use one common subsequence for a countable exhaustion of $\Omega$,
the density Young measure, all integer caps and their bounded products,
and the scalar truncated entropies below. Subsequent extractions
preserve these identifications.

\section{Truncated effective viscous flux}\label{sec:flux}

We work on $Q=(0,T)\times\Omega$ with the sequence and limits from
Section~\ref{sec:momentum}, where momentum and convection have been
identified as $m=\rho u$ and $F=\rho u\otimes u$. Strong convergence
of the density is not yet available. For a fixed cap level $k$, we
prove the effective-flux identity below using only the distributional
limit equations and the approximate renormalized continuity equations;
the source $S_k$ will remain unidentified.

Choose a smooth, increasing, concave function $T:[0,\infty)\to[0,2]$
such that $T(s)=s$ for $0\le s\le1$ and $T(s)=2$ for $s\ge3$. For
integers $k\ge1$, set
\begin{equation}\label{flux:cap}
 T_k(s)=kT(s/k),\qquad a_k(s)=sT_k'(s)-T_k(s).
\end{equation}
Thus $0\le T_k'\le1$, $T_k(s)\le\min\{s,2k\}$, and
\begin{equation}\label{flux:cap-bounds}
 a_k(s)\le0,\qquad
 |a_k(s)|\le T_k(s)\mathbf1_{\{s>k\}}\le s\mathbf1_{\{s>k\}}.
\end{equation}
The cap parameter remains fixed until the end of the proof below. Write
\[
 q_n=T_k(\rho_n),\quad d_n=\operatorname{div}u_n,\quad
 m_n=\rho_nu_n,\quad F_{n,ij}=m_{n,i}u_{n,j},\quad
 p_n=\rho_n+\delta_n\rho_n^\beta.
\]
After refining the subsequence simultaneously for all integer caps,

\begin{equation}\label{flux:weak-caps}
\begin{aligned}
 q_n&\overset{*}{\rightharpoonup}t_k
       \quad\text{in }L^\infty(Q),
 &\text{ and }\;\;\; \rho_nq_n&\rightharpoonup P_k
       \quad\text{in }L^1(Q),\\
 q_nd_n&\rightharpoonup h_k,
 &\text{ and }\;\;\; a_k(\rho_n)d_n&\rightharpoonup S_k
       \quad\text{both in }L^2(Q).
\end{aligned}
\end{equation}

We also extract weak-$*$ limits of $q_n^2$ and the bounded cap functions
needed below. The $L^1$ weak compactness of $\rho_nq_n$ follows from
$0\le\rho_nq_n\le2k\rho_n$ and uniform integrability of the density.
In particular, $P_k$ is a function, with no concentration measure.
These symbols do not yet identify nonlinear functions of $\rho$.

\begin{proposition}[Fixed-cap flux identity]\label{flux:proposition}
For every fixed integer $k\ge1$, let $t_k,P_k,h_k$ be the limits in
\eqref{flux:weak-caps}. Then, with $\nu=2\mu+\lambda>0$,
\begin{equation}\label{flux:identity}
 P_k-\rho t_k
   =\nu\bigl(h_k-t_k\operatorname{div}u\bigr)
       \qquad\text{almost everywhere in }Q,
\end{equation}
and every term belongs to $L^1(Q)$.

\end{proposition}

\noindent\textit{Proof of Proposition~\ref{flux:proposition}.}
The proof occupies the following four subsections.

\subsection{Localized truncation equations and their potentials}

\paragraph{The localized cap equation.}
Fix $\phi\in C_c^\infty(Q)$ and choose $\chi\in C_c^\infty(\Omega)$,
with $0\le\chi\le1$, equal to one near the spatial projection of
$\operatorname{supp}\phi$. This cutoff gives the cap compact spatial
support for the whole-space potential and leaves it unchanged near
the support of the test function. Extend $u_n$ and $b_n:=\chi q_n$ by zero
to $\mathbb R^2$. Applying \eqref{app:fixed-renormalized} to
$(\rho_n,u_n)$ with $b=T_k$ and multiplying by $\chi$, we obtain
in $\mathcal D'((0,T)\times\mathbb R^2)$
\begin{equation}\label{flux:localized-cap}
 \partial_tb_n+\operatorname{div}(b_nu_n)=g_n,
 \qquad g_n=-\chi a_k(\rho_n)d_n+q_nu_n\cdot\nabla\chi.
\end{equation}
Since $|q_n|,|a_k(\rho_n)|\le2k$ and $u_n$ is bounded in
$L^2(0,T;H^1_0(\Omega))$, both terms of $g_n$ are uniformly bounded
in $L^2((0,T)\times\mathbb R^2)$ for fixed $k$.
The cutoff term is supported away from $\operatorname{supp}\phi$,
but its nonlocal potential will be retained.

\paragraph{Cap compactness and the source limit.}
We first obtain the compactness needed to pass the localized cap
equation to the limit. For each finite $1<r<\infty$, boundedness and fixed spatial support
of $b_n$ give relative compactness of its spatial values in
$W^{-1,r}(\mathbb R^2)$, by duality and the 2D Rellich theorem.
Moreover, $b_nu_n$ is bounded in $L^2_tL^r_x$, and $g_n$ is bounded in
$L^2_tW^{-1,r}_x$. Thus \eqref{flux:localized-cap} supplies the time bound
\[
 \|\partial_tb_n\|_{L^2(0,T;W^{-1,r}(\mathbb R^2))}
 \le C_{k,\chi,r}.
\]
Arzel\`a--Ascoli and a countable diagonal extraction, with the limit
identified by $q_n\weakstar t_k$, give
\begin{equation}\label{flux:cap-compactness}
 b_n\longrightarrow b:=\chi t_k
 \quad\text{in }C([0,T];W^{-1,r}(\mathbb R^2)),
 \qquad 1<r<\infty.
\end{equation}
In particular, $b_n\to b$ strongly in $L^2_tH^{-1}_x$.

This compactness also determines the transport part of the source.
Pairing this convergence against smooth multiples of
$u_n\rightharpoonup u$ in $L^2_tH^1_x$ gives $b_nu_n\to bu$
in distributions. Repeating the argument with a cutoff equal to
one near $\operatorname{supp}\nabla\chi$ identifies
$q_nu_n\cdot\nabla\chi\to t_ku\cdot\nabla\chi$ in distributions.
The uniform $L^2$ source bound and
$a_k(\rho_n)d_n\rightharpoonup S_k$ therefore imply
\begin{equation}\label{flux:limiting-cap}
 \begin{aligned}
 g_n&\rightharpoonup g:=-\chi S_k+t_ku\cdot\nabla\chi
 \quad\quad\text{ in }L^2((0,T)\times\mathbb R^2),\\
& \partial_tb+\operatorname{div}(bu)=g
\;\;\;\;\quad\quad \quad\quad\text{ in }\mathcal D'((0,T)\times\mathbb R^2).
 \end{aligned}
\end{equation}

\paragraph{The inverse-divergence potential.}
We next prepare the potential that will enter the momentum test.
On $\mathbb R^2$, let $G(x)=(2\pi)^{-1}\log|x|$ and define
\begin{equation}\label{flux:operators}
 A_iv=\partial_iG*v,\qquad R_{ij}v=\partial_jA_iv,
 \qquad \operatorname{div}Av=v.
\end{equation}
The operator $R_{ij}$ has Fourier multiplier $\xi_i\xi_j/|\xi|^2$,
is self-adjoint on $L^2$, and is bounded on $L^r$ for
$1<r<\infty$. For $v$ supported in a fixed ball $B_0$, local
integrability of $\nabla G$ and these bounds give, for every ball $B$,
\[
 \|Av\|_{W^{1,r}(B)}
 \le C_{B,B_0,r}\|v\|_{L^r(\mathbb R^2)}.
\]
Thus the potential is defined locally without a mean-zero condition;
no global $L^2$ bound for $Av$ is required.

Put $A_n=Ab_n$ and $A=Ab$. The local potential estimate and
\eqref{flux:localized-cap} give
\begin{equation}\label{flux:potential-bounds}
 \begin{aligned}
 \|A_n\|_{L^\infty(0,T;W^{1,r}(B))}&\le C_{k,\chi,B,r},\\
 \partial_tA_{n,i}&=-R_{ij}(b_nu_{n,j})+A_ig_n,
 \end{aligned}
\end{equation}
where $1<r<\infty$ and repeated indices are summed.
The source potential $Ag_n$ is bounded in $L^2_tH^1(B)$, hence
in $L^2_tL^r(B)$ for every finite $r>1$ by the 2D Sobolev embedding.
The same $L^2_tL^r(B)$ bound holds for $R_{ij}(b_nu_{n,j})$.
Consequently, $\partial_tA_n$ is bounded in $L^2_tL^r(B)$.

Choose $r>2$. The compact embedding
$W^{1,r}(B)\Subset C(\overline B)$ and Ehrling's inequality upgrade
the uniform $1/2$-H\"older time modulus in $L^r(B)$ to equicontinuity
in $C(\overline B)$. Arzel\`a--Ascoli then gives compactness in
$C([0,T];C(\overline B))$. The distributional convergence
$Ab_n\to Ab$ identifies the limit, so, after a common diagonal
extraction over balls,
\begin{equation}\label{flux:potential-compactness}
 A_n\longrightarrow A
 \quad\text{in }C([0,T];C(\overline B)).
\end{equation}
Together, \eqref{flux:cap-compactness} and
\eqref{flux:potential-compactness} provide the cap and potential
compactness for the momentum test. The Riesz-transform estimates
above use finite exponents and require no $L^\infty$ bound.

\subsection{The localized momentum identity}

\begin{lemma}[Localized momentum identity]\label{flux:testing-lemma}
Fix $k$, $\phi$, and $\chi$ as in Subsection~4.1, and retain its notation.
Set $\mathbb S_n=\mathbb S(u_n)$ and
$\mathcal C_{n,i}=[u_{n,j},R_{ij}]b_n$.
The test $\Psi_{n,i}=\phi A_i[\chi T_k(\rho_n)]$ is admissible
by regularization. After removing the regularization, it yields
\begin{align*}
 \int_Q\phi(p_n-\nu d_n)b_n
 ={}&-\int_Q\phi m_n\cdot\mathcal C_n
      -\int_Q\phi m_n\cdot Ag_n
      -\int_Q(\partial_t\phi)m_n\cdot A_n\\
    &-\int_Q(\partial_j\phi)F_{n,ij}A_{n,i}
      -\int_Qp_n\nabla\phi\cdot A_n
      +\int_Q\mathbb S_{n,ij}(\partial_j\phi)A_{n,i}\\
    &+2\mu\int_Q u_{n,i}
       \bigl((\partial_i\phi)b_n-(\partial_j\phi)R_{ij}b_n\bigr).
\end{align*}
This is the localized momentum identity \eqref{flux:full} below.
\end{lemma}

\begin{proof}
\noindent\textbf{Admissibility of the momentum test.}
We first justify the approximate and limiting tests by mollification;
the inertial and stress calculations will then give the full identity.
For a smooth compactly supported vector test $\Psi$, the approximate
momentum equation corresponding to \eqref{app:fixedsystem}, with
$(r,v,\delta)=(\rho_n,u_n,\delta_n)$, reads
\begin{equation}\label{flux:approximate-momentum}
 0=\int_Q \left[ m_{n,i}\partial_t\Psi_i
    +F_{n,ij}\partial_j\Psi_i+p_n\partial_i\Psi_i
    -\mathbb S_{n,ij}\partial_j\Psi_i \right].
\end{equation}

Let $\Psi_{n,i}=\phi A_i[\chi T_k(\rho_n)]=\phi A_{n,i}$ and
$A_n^\varepsilon=A[b_n^\varepsilon]$, where the mollification is
specified below. Testing \eqref{flux:approximate-momentum} with
$\Psi_n^\varepsilon=\phi A_n^\varepsilon$ gives
\begin{equation}\label{flux:regularized-momentum-test}
 \begin{aligned}
 0=\int_Q \Big[&m_{n,i}\partial_t(\phi A_{n,i}^\varepsilon)
   +F_{n,ij}\partial_j(\phi A_{n,i}^\varepsilon)
   +p_n\partial_i(\phi A_{n,i}^\varepsilon)-\mathbb S_{n,ij}\partial_j(\phi A_{n,i}^\varepsilon)\Big].
 \end{aligned}
\end{equation}
Here, for fixed $n$ and $k$, we let $v^\varepsilon$ denote
standard mollification of $v$ in space and time with nonnegative
unit-mass kernels. The localized quantities $b_n$, $b_nu_n$, and $g_n$
are extended by zero in space before convolution; $\varepsilon$ is small
enough that time convolution near $\operatorname{supp}\phi$ stays inside
$(0,T)$. Mollifying \eqref{flux:localized-cap} and using
\eqref{flux:operators} gives
\begin{equation}\label{flux:regularized-cap}
 \begin{aligned}
 \partial_tb_n^\varepsilon
 +\operatorname{div}((b_nu_n)^\varepsilon)&=g_n^\varepsilon,\\
 \partial_tA_{n,i}^\varepsilon
 &=-R_{ij}[(b_nu_{n,j})^\varepsilon]+A_i[g_n^\varepsilon].
 \end{aligned}
\end{equation}
Thus $\Psi_n^\varepsilon$ is a smooth, compactly supported momentum
test whose time derivative is determined by
\eqref{flux:regularized-cap}.

To remove the smoothing at fixed $n$ and $k$, we use the following
pairings for the time, convective, and viscous terms, respectively.
On an interior cylinder $U=I\times D\Subset Q$ containing
$\operatorname{supp}\phi$, they follow from \eqref{mom:mixed} and
\eqref{flux:potential-bounds}:
\begin{align*}
 &m_n\in L^3_tL^{96/95}_x,
       &&\partial_tA_n\in L^2_tL^{96}_x,\\
 &F_n\in L^{6/5}_tL^{192/191}_x,
       &&\nabla A_n\in L^6_tL^{192}_x,\\
 &\mathbb S(u_n)\in L^2,
       &&\nabla\Psi_n\in L^2.
\end{align*}
All spaces in this display are restricted to $U$.
For fixed $n$ and $k$, as $\varepsilon\to0$, we have
\[
 \partial_tA_n^\varepsilon\longrightarrow\partial_tA_n
    \quad\text{in }L^2(I;L^{96}(D)),\qquad
 \nabla A_n^\varepsilon\longrightarrow\nabla A_n
    \quad\text{in }L^6(I;L^{192}(D)),
\]
and
\[
 \nabla\Psi_n^\varepsilon\longrightarrow\nabla\Psi_n
    \quad\text{in }L^2(U).
\]
H\"older's inequality and the convergences above pass the principal
time and convective contributions
\[
 \int_Q \phi m_{n,i}\partial_tA_{n,i}^\varepsilon,
 \qquad
 \int_Q \phi F_{n,ij}\partial_jA_{n,i}^\varepsilon,
\]
as $\varepsilon\to0$, while
$\nabla\Psi_n^\varepsilon\to\nabla\Psi_n$ in $L^2(U)$ passes the full
viscous term
\[
 -\int_Q\mathbb S_{n,ij}\partial_j(\phi A_{n,i}^\varepsilon).
\]
The cutoff contributions from the time and convective terms,
\[
 \int_Q(\partial_t\phi)m_n\cdot A_n^\varepsilon,
 \qquad
 \int_Q(\partial_j\phi)F_{n,ij}A_{n,i}^\varepsilon,
\]
converge by the uniform convergence
$A_n^\varepsilon\to A_n$ on $\overline U$.

For the pressure term, \eqref{flux:operators} gives the decomposition
\[
 \int_Q p_n\operatorname{div}\Psi_n^\varepsilon
 =\int_Q p_n\phi b_n^\varepsilon
  +\int_Q p_n\nabla\phi\cdot A_n^\varepsilon.
\]
At fixed $n$, $p_n\in L^1(Q)$, while
$0\le b_n^\varepsilon\le2k$ and $b_n^\varepsilon\to b_n$
almost everywhere. Dominated convergence therefore passes the
principal pressure contribution
$\int_Q p_n\phi b_n^\varepsilon$. The cutoff pressure contribution
$\int_Q p_n\nabla\phi\cdot A_n^\varepsilon$ converges because
$A_n^\varepsilon\to A_n$ uniformly on $\overline U$.
This justifies removing the smoothing in the approximate momentum
identity at fixed $n$ and $k$.

The limiting test is justified by the same procedure, independently
of the sequence limit. For $\Psi=\phi Ab$, we use the interior momentum
identity \eqref{weak:mom}, established in
Section~\ref{sec:momentum}, and mollify the limiting truncation
equation \eqref{flux:limiting-cap}.
The identities \eqref{flux:regularized-cap} then hold with
$(b_n,b_nu_n,g_n)$ replaced by $(b,bu,g)$. Writing
$A^\varepsilon=A[b^\varepsilon]$ and
$\Psi^\varepsilon=\phi A^\varepsilon$, the same finite-exponent
estimates pass the principal time and convective contributions
\[
 \int_Q\phi m_i\partial_tA_i^\varepsilon,
 \qquad
 \int_Q\phi F_{ij}\partial_jA_i^\varepsilon,
\]
and the $L^2$ convergence of $\nabla\Psi^\varepsilon$ passes the full
viscous term
\[
 -\int_Q\mathbb S_{ij}(u)\partial_j\Psi_i^\varepsilon.
\]
The corresponding time and convective cutoff contributions, containing
$(\partial_t\phi)A^\varepsilon$ and
$(\nabla\phi)A^\varepsilon$, pass by uniform local convergence of
$A^\varepsilon$ to $Ab$.

For the principal pressure contribution, since $\rho\in L^1(U)$,
$0\le b^\varepsilon\le2k$, and $b^\varepsilon\to b$ almost everywhere,
dominated convergence gives
\[
 \int_Q\rho\phi b^\varepsilon\longrightarrow\int_Q\rho\phi b.
\]
The cutoff pressure contribution
$\int_Q\rho\nabla\phi\cdot A^\varepsilon$ converges by the same
uniform local convergence and the local $L^1$ integrability of
$\rho$. This completes the justification of both momentum tests;
we now compute the approximate identity before passing $n\to\infty$.

\noindent\textbf{The inertial terms.}
Let $\mathbb S_n=\mathbb S(u_n)$ and
\[
 \mathcal C_{n,i}=[u_{n,j},R_{ij}]b_n
       :=u_{n,j}R_{ij}b_n-R_{ij}(u_{n,j}b_n).
\]
Using \eqref{flux:potential-bounds}, the principal time and convective
contributions
$\int_Q\phi m_{n,i}\partial_tA_{n,i}$ and
$\int_Q\phi F_{n,ij}\partial_jA_{n,i}$ combine exactly into
\begin{equation}\label{flux:inertia}
 \int_Q\phi m_{n,i}
       \{-R_{ij}(b_nu_{n,j})+A_ig_n+u_{n,j}R_{ij}b_n\}
 =\int_Q\phi m_n\cdot\mathcal C_n+
       \int_Q\phi m_n\cdot Ag_n.
\end{equation}
This grouping is performed before taking $n\to\infty$.

\noindent\textbf{The stress contribution.}
For the stress, symmetry of $R_{ij}$ and integration by parts give
\begin{equation}\label{flux:stress}
 \int_Q\phi\mathbb S_{n,ij}R_{ij}b_n
 =\nu\int_Q\phi d_nb_n
   +2\mu\int_Q u_{n,i}
       \bigl((\partial_i\phi)b_n-(\partial_j\phi)R_{ij}b_n\bigr).
\end{equation}
Indeed, each symmetric-gradient term equals
$\int_Q\phi d_nb_n+\int_Q u_{n,i}(\partial_i\phi)b_n
-\int_Q u_{n,i}(\partial_j\phi)R_{ij}b_n$;
the bulk-viscosity term is $\lambda\int_Q\phi d_nb_n$.

\noindent\textbf{The full identity.}
Substituting the grouped inertia and stress formula into the momentum
identity yields the formula to be passed to the limit:
\begin{align}
 \int_Q\phi(p_n-\nu d_n)b_n
 ={}&-\int_Q\phi m_n\cdot\mathcal C_n
      -\int_Q\phi m_n\cdot Ag_n
      -\int_Q(\partial_t\phi)m_n\cdot A_n \notag\\
    &-\int_Q(\partial_j\phi)F_{n,ij}A_{n,i}
      -\int_Qp_n\nabla\phi\cdot A_n
      +\int_Q\mathbb S_{n,ij}(\partial_j\phi)A_{n,i}\notag\\
    &+2\mu\int_Q u_{n,i}
       \bigl((\partial_i\phi)b_n-(\partial_j\phi)R_{ij}b_n\bigr).
       \label{flux:full}
\end{align}
The derivatives of $\chi$ enter through the source $g_n$ in
\eqref{flux:localized-cap}, whose potential is retained in the first
line. Since $\chi=1$ on $\operatorname{supp}\phi$, the left side contains
the original, unlocalized cap there. This proves the stated identity.
\end{proof}

\subsection{Passage of the nonlocal terms}

\begin{lemma}[Passage of the nonlocal terms]\label{flux:nonlocal-lemma}
Under the notation of Subsections~4.1--4.2, with $k$ fixed, as $n\to\infty$,
\[
 \int_Q\phi m_n\cdot\mathcal C_n
 \longrightarrow\int_Q\phi m\cdot\mathcal C,
 \qquad
 \int_Q\phi m_n\cdot Ag_n
 \longrightarrow\int_Q\phi m\cdot Ag,
\]
where $\mathcal C_{n,i}=[u_{n,j},R_{ij}]b_n$,
$\mathcal C_i=[u_j,R_{ij}]b$, and
\[
 g=-\chi S_k+t_ku\cdot\nabla\chi.
\]
The source $S_k$ remains the weak limit in \eqref{flux:weak-caps};
it is not identified with $a_k(\rho)\operatorname{div}u$.
\end{lemma}

\begin{proof}
Both passages use the same negative-Sobolev duality: strong compactness
of the localized momentum is paired with weak convergence of the
commutator or source potential.

\noindent\textbf{The commutator term.}
We use the following first-derivative commutator estimate of
Coifman--Meyer type \cite{CM,Lions}:
\begin{equation}\label{flux:commutator-estimate}
 \|\nabla[v,R_{ij}]w\|_{L^r(\mathbb R^2)}
 \le C_{p,q,r}\|\nabla v\|_{L^p(\mathbb R^2)}
                   \|w\|_{L^q(\mathbb R^2)},
 \quad \frac1r=\frac1p+\frac1q,
 \quad 1<p,q,r<\infty.
\end{equation}
Taking $p=2$, $q=192$, and hence $r=192/97$ in
\eqref{flux:commutator-estimate}, the undifferentiated norm is bounded
with the same exponents:
\[
 \|u_{n,j}R_{ij}b_n\|_{L^{\frac{192}{97}}}
  +\|R_{ij}(u_{n,j}b_n)\|_{L^{\frac{192}{97}}}
 \le C\|u_n\|_{L^2}\|b_n\|_{L^{192}}.
\]
Poincar\'e's inequality and \eqref{flux:commutator-estimate} therefore give
\begin{equation}\label{flux:commutator-bound}
 \|\mathcal C_n\|_{L^2_tW^{1,\frac{192}{97}}(\mathbb R^2)}
       \le C_{k,\chi,\Omega}\|\nabla u_n\|_{L^2(Q)}.
\end{equation}

To identify the weak limit of the commutator, we use the cap
compactness, without strong $L^1$ density convergence. For any scalar
$\eta\in C_c^\infty((0,T)\times\mathbb R^2)$,
self-adjointness gives, with the time integrations understood,
\[
 \int\eta u_{n,j}R_{ij}b_n
       =\langle b_n,R_{ij}(\eta u_{n,j})\rangle,
 \qquad
 \int\eta R_{ij}(b_nu_{n,j})
       =\langle b_n,u_{n,j}R_{ij}\eta\rangle.
\]

Multiplication by the fixed smooth function \(\eta\) is bounded on
\(H^1(\mathbb R^2)\), and \(R_{ij}\) is bounded on the same space.
The zero extensions satisfy \(u_n\rightharpoonup u\) in
\(L^2(0,T;H^1(\mathbb R^2))\), and hence
\[
 R_{ij}(\eta u_{n,j})
 \rightharpoonup
 R_{ij}(\eta u_j)
 \qquad\text{in }L^2(0,T;H^1(\mathbb R^2)).
\]
Moreover, \(R_{ij}\eta\) and its first spatial derivatives are bounded
on \(\mathbb R^2\), uniformly in time, and hence
\[
 u_{n,j}R_{ij}\eta
 \rightharpoonup
 u_jR_{ij}\eta
 \qquad\text{in }L^2(0,T;H^1(\mathbb R^2)).
\]
On the other hand, \eqref{flux:cap-compactness} gives
\[
 b_n\longrightarrow b
 \qquad\text{strongly in }L^2(0,T;H^{-1}(\mathbb R^2)).
\]
The \(H^{-1}(\mathbb R^2)\)--\(H^1(\mathbb R^2)\) duality pairings above therefore pass
to the limit, giving
\[
 \mathcal C_n
 \longrightarrow
 [u_j,R_{ij}]b
 \qquad\text{in }\mathcal D'((0,T)\times\mathbb R^2).
\]

Combining this identification with \eqref{flux:commutator-bound} gives
\begin{equation}\label{flux:commutator-limit}
 \mathcal C_n\rightharpoonup
   \mathcal C:=[u_j,R_{ij}]b
       \quad\text{in }L^2_tW^{1,192/97}_x.
\end{equation}

We now pair this weak limit with the strongly convergent momentum.
Let $\zeta\in C_c^\infty(\Omega)$ equal one near the spatial support of
$\phi$, and choose a ball $B$ such that $\operatorname{supp}\zeta\Subset B$.
By Proposition~\ref{mom:compact}, in particular \eqref{mom:exact},
\[
 \zeta m_n\to\zeta m
       \quad\text{strongly in }L^2_tW^{-1,192/95}_x.
\]

Since \(\phi\) is compactly supported in \(B\), multiplication by
\(\phi\) maps \(W^{1,192/97}(B)\) continuously into
\(W^{1,192/97}_0(B)\). Hence \eqref{flux:commutator-limit} gives
\[
 \phi\mathcal C_n\rightharpoonup\phi\mathcal C
 \qquad\text{in }L^2(0,T;W^{1,192/97}_0(B)).
\]
The strong--weak pairing can therefore be written as
\begin{align*}
 \int_Q\phi m_n\cdot\mathcal C_n-
      \int_Q\phi m\cdot\mathcal C
 ={}&\langle\zeta(m_n-m),\phi\mathcal C_n\rangle\\
    &+\langle\zeta m,\phi(\mathcal C_n-\mathcal C)\rangle
       \longrightarrow0.
\end{align*}
This passes the principal commutator term of \eqref{flux:full}, with
dual exponents \(192/95\) and \(192/97\).

\noindent\textbf{The source-potential term.}
The same pairing applies to the source potential. By
\eqref{flux:limiting-cap} and the local potential estimate,
\[
 Ag_n\rightharpoonup Ag
 \qquad\text{in }L^2(0,T;H^1(B)),
\]
and therefore also
\[
 Ag_n\rightharpoonup Ag
 \qquad\text{in }L^2(0,T;W^{1,192/97}(B)).
\]
Since \(\phi\) is compactly supported in \(B\), multiplication by \(\phi\) yields
\[
 \phi Ag_n\rightharpoonup\phi Ag
 \qquad\text{in }L^2(0,T;W^{1,192/97}_0(B)).
\]
With the same cutoff \(\zeta\) and momentum convergence
\eqref{mom:exact}, we obtain
\[
\begin{aligned}
 \int_Q\phi m_n\cdot Ag_n-\int_Q\phi m\cdot Ag
 ={}&\langle\zeta(m_n-m),\phi Ag_n\rangle\\
    &+\langle\zeta m,\phi(Ag_n-Ag)\rangle
       \longrightarrow0.
\end{aligned}
\]
Hence
\begin{equation}\label{flux:source-passage}
 \int_Q\phi m_n\cdot Ag_n\longrightarrow
       \int_Q\phi m\cdot Ag .
\end{equation}
The source-potential term of \eqref{flux:full} thus passes by the same
duality as the commutator term. Throughout this passage,
\[
 g=-\chi S_k+t_ku\cdot\nabla\chi,
\]
and the source \(S_k\) has not been replaced by
\(a_k(\rho)\operatorname{div}u\).

\end{proof}

\subsection{Proof of Proposition 4.1}

We now pass to the limit in the remaining terms of the identity from
Lemma~\ref{flux:testing-lemma}, use Lemma~\ref{flux:nonlocal-lemma}
for the two nonlocal contributions, and compare the resulting formula
with the limiting momentum equation.

\begin{proof}[Completion of the proof of Proposition~\ref{flux:proposition}]

\noindent\textbf{Remaining cutoff and stress terms.}
The terms containing the
undifferentiated potential follow from its uniform convergence. By
\eqref{flux:potential-compactness}, $A_n\to A$ uniformly on every
compact subset of $Q$. Together with the already identified weak limits
of $m_n$, $F_n$, and $\mathbb S_n$, this gives
\[
 \int_Q(\partial_t\phi)m_n\cdot A_n
      \longrightarrow \int_Q(\partial_t\phi)m\cdot A,
\]
\[
 \int_Q(\partial_j\phi)F_{n,ij}A_{n,i}
      \longrightarrow \int_Q(\partial_j\phi)F_{ij}A_i,
 \qquad
 \int_Q\mathbb S_{n,ij}(\partial_j\phi)A_{n,i}
      \longrightarrow \int_Q\mathbb S_{ij}(u)(\partial_j\phi)A_i.
\]
The two cutoff terms from the stress formula, in the last line of
\eqref{flux:full}, use cap compactness. For the first one, \eqref{flux:cap-compactness} and
$u_n\rightharpoonup u$ in $L^2_tH^1_x$ give
\[
 \int_Q u_{n,i}(\partial_i\phi)b_n
 =\int_0^T\langle b_n,u_{n,i}\partial_i\phi\rangle\,dt
 \longrightarrow
 \int_0^T\langle b,u_i\partial_i\phi\rangle\,dt.
\]
For the Riesz-transform term, self-adjointness gives
\[
 \int_Q u_{n,i}(\partial_j\phi)R_{ij}b_n
 =\int_0^T\langle b_n,R_{ij}(u_{n,i}\partial_j\phi)\rangle\,dt.
\]
Since
\[
 R_{ij}(u_{n,i}\partial_j\phi)
 \rightharpoonup R_{ij}(u_i\partial_j\phi)
 \qquad\text{in }L^2_tH^1_x,
\]
the strong $L^2_tH^{-1}_x$ convergence in
\eqref{flux:cap-compactness} yields
\[
 \int_Q u_{n,i}(\partial_j\phi)R_{ij}b_n
 \longrightarrow
 \int_Q u_i(\partial_j\phi)R_{ij}b.
\]

\noindent\textbf{Pressure and artificial pressure.}
For the pressure terms, we use the local artificial-pressure
convergence from Lemma~\ref{mom:artificial}, in
particular \eqref{mom:AP},
\[
 \delta_n\rho_n^\beta\longrightarrow0
 \qquad\text{strongly in }L^1((0,T)\times D)
\]
for every $D\Subset\Omega$. Hence, since $|b_n|\le2k$,
\[
 \delta_n\rho_n^\beta b_n\longrightarrow0
 \qquad\text{in }L^1((0,T)\times D).
\]
For the pressure cutoff term in \eqref{flux:full}, write
\[
 \int_Q p_n\nabla\phi\cdot A_n
 =\int_Q \rho_n\nabla\phi\cdot A_n
  +\int_Q\delta_n\rho_n^\beta\nabla\phi\cdot A_n.
\]
The first term converges to
$\int_Q\rho\nabla\phi\cdot A$ by the weak $L^1$ convergence of
$\rho_n$ and the uniform convergence of $A_n$ on
$\operatorname{supp}\nabla\phi$; the second tends to zero by
\eqref{mom:AP} and the same local uniform bound. Thus
\[
 \int_Q p_n\nabla\phi\cdot A_n
 \longrightarrow \int_Q\rho\nabla\phi\cdot A.
\]
For the principal pressure--viscous term on the left-hand side of
\eqref{flux:full}, the relations in \eqref{flux:weak-caps}, the fact
that $b_n=q_n$ on $\operatorname{supp}\phi$, and the local artificial-
pressure convergence above give
\begin{equation}\label{flux:left-passage}
 \int_Q\phi(p_n-\nu d_n)b_n
 \longrightarrow
 \int_Q\phi(P_k-\nu h_k).
\end{equation}

\noindent\textbf{Comparison with the limiting momentum equation.}
The preceding passages and Lemma~\ref{flux:nonlocal-lemma}
give the limit of \eqref{flux:full}. We compare it with the identity
obtained by testing the limiting momentum equation with $\phi A$.

The mollification argument in Lemma~\ref{flux:testing-lemma}
justifies this test. We compute $\partial_t A$ from
\eqref{flux:limiting-cap} and repeat the expansions in
\eqref{flux:inertia}--\eqref{flux:stress}, with $p=\rho$ and
$\mathbb S=\mathbb S(u)$. The resulting right-hand side is
exactly the limit obtained above. In particular, both identities
contain the same source-potential term, with
$g=-\chi S_k+t_ku\cdot\nabla\chi$. This term therefore cancels
in the comparison, without any identification of $S_k$.

Since $b=\chi t_k$ and $\chi=1$ on $\operatorname{supp}\phi$,
the principal pressure--viscous term of the limiting identity is
\[
 \int_Q\phi\bigl(\rho-\nu\operatorname{div}u\bigr)b
 =
 \int_Q\phi\bigl(\rho-\nu\operatorname{div}u\bigr)t_k.
\]
Comparison with \eqref{flux:left-passage} thus gives
\[
 \int_Q\phi(P_k-\nu h_k)
 =
 \int_Q\phi\bigl(\rho-\nu\operatorname{div}u\bigr)t_k
 \qquad\text{for every }\phi\in C_c^\infty(Q).
\]
This is \eqref{flux:identity} in $\mathcal D'(Q)$.
Since $P_k,\rho t_k\in L^1(Q)$ and
$h_k,t_k\operatorname{div}u\in L^2(Q)\subset L^1(Q)$,
the identity holds almost everywhere in $Q$.
This proves Proposition~\ref{flux:proposition}.
\end{proof}

\begin{remark}\label{flux:order}
The argument is carried out with $k$ fixed. Neither strong density
convergence nor limiting renormalization is used.
Source removal, the limit $k\to\infty$, and the nonlinear initial
trace are treated below. The interior flux identity alone does not
determine the time-zero entropy trace.
\end{remark}

\section{Quadratic oscillation and endpoint cap-source control}\label{sec:variance}

We develop a new argument for removing the truncation source
at the endpoint $\gamma=1$. The key ingredient is a single
integrable variance majorant, independent of the truncation
parameter. We start from the classical
effective-flux estimate for concave truncations
\cite{Lions,Feireisl,FNP}. We express this estimate in centered
covariance--variance form and prove that the resulting density
variance is integrable. This variance provides a
truncation-independent bound for the second moment of the
truncation residual on regions of bounded cap mean.

\subsection{Classical quadratic oscillation in centered form}

We use probabilistic notation for the density Young measure. At almost every
$(t,x)\in Q_T$, let $X$ denote the coordinate variable on $[0,\infty)$
with law $\eta_{t,x}$. For every integrable function $f$, we write
\[
 \E f(X):=\int_{[0,\infty)} f(z)\,d\eta_{t,x}(z).
\]
Accordingly, for random variables $Y$ and $Z$ for which the expressions
below are well defined, we set
\[
 \Cov(Y,Z):=\E\bigl[(Y-\E Y)(Z-\E Z)\bigr],
 \qquad
 \Var(Y):=\E|Y-\E Y|^2\in[0,\infty].
\]
The following elementary lemma is purely measure-theoretic and will later
be applied pointwise to the density Young measure $\eta_{t,x}$.

\begin{lemma}[Centered cap oscillation]\label{var:abstract}
Let $X\ge0$ with $\E X=r<\infty$,  and  $T_k$ be the truncations of
Section~\ref{sec:flux}. Let
\[
 t_k=\E T_k(X),\qquad C_k=\Cov(X,T_k(X)),\qquad
 V_k=\Var(T_k(X)),\qquad V=\E|X-r|^2\in[0,\infty].
\]
Then
\begin{equation}\label{var:monotone}
 0\le V_k\le C_k,\qquad C_k\uparrow V\quad(k\to\infty).
\end{equation}
If $V<\infty$, then
\begin{equation}\label{var:jensen}
 0\le T_k(r)-t_k,\qquad |T_k(r)-t_k|^2\le V.
\end{equation}
\end{lemma}

\begin{proof}

Let \(Y\) be an independent copy of \(X\). Expanding the products and using independence and identical distribution, we obtain
\begin{equation}
\label{name of Ck}
\begin{aligned}
\frac12\E\bigl[(X-Y)(T_k(X)-T_k(Y))\bigr]
&=\E[XT_k(X)]-\E X\,\E T_k(X)\\
&=\Cov(X,T_k(X)),
\end{aligned}
\end{equation}
and similarly
\[
\frac12\E|T_k(X)-T_k(Y)|^2
=\E|T_k(X)|^2-\bigl|\E T_k(X)\bigr|^2
=\Var(T_k(X)).
\]
Thus
\begin{align}
 C_k&=\frac12\E[(X-Y)(T_k(X)-T_k(Y))],\label{var:copies}\\
 V_k&=\frac12\E|T_k(X)-T_k(Y)|^2.
\end{align}
The integrand on the right-hand side of \eqref{var:copies} is integrable because \(T_k\) is bounded and \(\E X<\infty\).
 We denote that 
\[
 a=X-Y,\qquad b=T_k(X)-T_k(Y).
\]
Since $T_k$ is increasing, $ab\ge0$, while the one-Lipschitz property gives
$|b|\le|a|$. Hence, pointwise,
\(
 0\le b^2\le ab,
\)
and taking expectations yields $0\le V_k\le C_k$.

For $x>y$,
\[
 T_k(x)-T_k(y)=\int_y^x T'(s/k)\,ds\uparrow x-y.
\]
Indeed, concavity of $T$ makes $T'$ nonincreasing, so for each fixed
$s\ge0$ the quantity $T'(s/k)$ increases with $k$ and eventually equals
$1$. By symmetry, for every $x,y\ge0$,
\[
 (x-y)(T_k(x)-T_k(y))\uparrow |x-y|^2.
\]
Thus monotone convergence in \eqref{var:copies} gives
\[
 C_k\uparrow\frac12\E|X-Y|^2.
\]
If $\E X^2<\infty$, independence gives
\[
 \frac12\E|X-Y|^2=\E X^2-(\E X)^2=V.
\]

If \(\E X^2=\infty\), then we have 
\(
\E|X-Y|^2=\infty.
\)
In fact, 
since \(X\) and \(Y\) are independent copies and \(\E X=\E Y=r<\infty\), one obtains that 
\[
\E(XY)=\E X\,\E Y=r^2<\infty.
\]
If instead \(\E|X-Y|^2<\infty\), then
we would obtain
\[
\E(X^2+Y^2)
=
\E|X-Y|^2+2\E(XY)
<\infty.
\]
But \(X\) and \(Y\) are identically distributed, so
\(
\E(X^2+Y^2)=2\E X^2=\infty.
\) Thus, this is
a contradiction to \(\E X^2=\infty\). Hence we have
\(
\E|X-Y|^2=\infty.
\)

We next show that \(V=\infty\). Indeed, if
\(
V=\E|X-r|^2<\infty,
\)
then
\[\E
X^2=\E( |(X-r)+r|^2)
\le \E ( 2|X-r|^2+2r^2)\leq 2V+2r^2<+\infty,
\]
again a contradiction. Thus \(V=\infty\). Consequently,
\[
\frac12\E|X-Y|^2=V=\infty,
\]
and \eqref{var:monotone} holds also in the extended sense.


Finally, concavity of \(T_k\) and Jensen's inequality give
\[
t_k=\E T_k(X)\le T_k(\E X)=T_k(r),
\]
and hence
\[
T_k(r)-t_k
=\E\bigl[T_k(r)-T_k(X)\bigr]\ge0.
\]
Applying Jensen's inequality to the convex function \(z\mapsto |z|^2\), and then using the one-Lipschitz property of \(T_k\), we obtain
\[
|T_k(r)-t_k|^2
\le \E|T_k(r)-T_k(X)|^2
\le \E|r-X|^2=V.
\]

\end{proof}

We now apply Lemma~\ref{var:abstract} pointwise to the density Young
measure. Uniform integrability of $\rho_n$ gives $\E X=\rho(t,x)$ and,
with the notation introduced above,
\begin{equation}\label{var:representation}
 t_k=\E T_k(X),\qquad
 P_k=\E[XT_k(X)],\qquad
 T_k(\rho_n)^2\weakstar\E T_k(X)^2
 \quad\text{in }L^\infty(Q_T).
\end{equation}
Here the identification of $P_k$ contains no concentration term, since
$0\le \rho_nT_k(\rho_n)\le 2k\rho_n$ and $(\rho_n)_n$ is uniformly
integrable.
\subsection{Integrable centered variance}
We now prove that the centered density variance is integrable.
The fixed-cap effective-flux identity, together with
Lemma~\ref{var:abstract}, yields a bound uniform in the truncation parameter $k$.

\begin{proposition}[Integrable oscillation variance]\label{var:bound}
Let $d_n=\operatorname{div}u_n$, $d=\operatorname{div}u$, and
$D_*:=\sup_n\int_{Q_T}|d_n|^2$. Then
\begin{equation}\label{var:global}
 \int_{Q_T}V\le\nu^2D_*,\qquad \nu=2\mu+\lambda.
\end{equation}
More precisely, for every fixed $k$,
\begin{equation}\label{var:estimate}
 \int_{Q_T}C_k
 \le\nu D_*^{1/2}\left(\int_{Q_T}V_k\right)^{1/2}.
\end{equation}
\end{proposition}
\begin{proof}
By \eqref{name of Ck} and  \eqref{var:representation}, we have 
\[
 C_k=\E[XT_k(X)]-\E X\,\E T_k(X)=P_k-\rho t_k.
\]
Hence the fixed-cap flux identity \eqref{flux:identity} gives
\[
 C_k=\nu(h_k-t_kd)
\]
as globally integrable functions. Using the weak convergence
$q_nd_n\rightharpoonup h_k$ from \eqref{flux:weak-caps}, together with
$d_n\rightharpoonup d$ in $L^2(Q_T)$ from \eqref{mom:weak}, we obtain
\begin{equation}\label{var:flux-integrated}
\begin{aligned}
 \int_{Q_T}C_k
 &=\nu\left(\int_{Q_T}h_k-\int_{Q_T}t_kd\right)\\
 &=\nu\lim_{n\to\infty}\int_{Q_T}d_n(T_k(\rho_n)-t_k).
\end{aligned}
\end{equation}

The first-moment convergence in \eqref{flux:weak-caps} and the
second-moment convergence in \eqref{var:representation} yield
\begin{equation}\label{var:cap-L2}
\begin{aligned}
 \lim_{n\to\infty}\int_{Q_T}|T_k(\rho_n)-t_k|^2
 &=\lim_{n\to\infty}\left(
    \int_{Q_T}T_k(\rho_n)^2
   -2\int_{Q_T}t_kT_k(\rho_n)
   +\int_{Q_T}t_k^2\right)\\
 &=\int_{Q_T}\bigl(\E T_k(X)^2-t_k^2\bigr)
 =\int_{Q_T}V_k.
\end{aligned}
\end{equation}


Since \(C_k\ge0\) applying the Cauchy--Schwarz inequality to \eqref{var:flux-integrated}, we have 
\[
\begin{aligned}
\int_{Q_T} C_k
&\le
\nu\limsup_{n\to\infty}
\left|
\int_{Q_T}
d_n\bigl(T_k(\rho_n)-t_k\bigr)
\right| \\
&\le
\nu\limsup_{n\to\infty}
\|d_n\|_{L^2(Q_T)}
\|T_k(\rho_n)-t_k\|_{L^2(Q_T)}
\\&\le
\nu D_*^{1/2}
\left(\int_{Q_T}V_k\right)^{1/2},
\end{aligned}
\]
where we used \eqref{var:cap-L2}, and 
\(
\|d_n\|_{L^2(Q_T)}\le D_*^{1/2}.
\)


 Since $V_k\le C_k$ by
Lemma~\ref{var:abstract},
\[
 \int_{Q_T}C_k
 \le \nu D_*^{1/2}
      \left(\int_{Q_T}C_k\right)^{1/2}.
\]
If $\int_{Q_T}C_k=0$, the desired bound is immediate; otherwise,
dividing by its square root gives
\(
 \int_{Q_T}C_k\le\nu^2D_*.
\)
Finally, $C_k\uparrow V$ by \eqref{var:monotone}, so the monotone
convergence theorem yields
\[
 \int_{Q_T}V
 =\lim_{k\to\infty}\int_{Q_T}C_k
 \le\nu^2D_*,
\]
which is \eqref{var:global}.
\end{proof}

\begin{remark}
The estimate above controls oscillations about the local barycenter; it
does not imply either $\rho\in L^2(Q_T)$ or
$\E X^2\in L^1(Q_T)$. Indeed, the Dirac Young measure
$\eta_{t,x}=\delta_{\rho(t,x)}$ has zero variance independently of
whether $\rho\in L^2(Q_T)$. On the other hand, for almost every
$(t,x)$, the finiteness of $V(t,x)$ and $\rho(t,x)$ gives
\[
 \E X^2=V+\rho^2<\infty.
\]
Thus a finite second moment is available pointwise in the Young
probability space, but no global $L^1(Q_T)$ bound for $\E X^2$ is
asserted.
\end{remark}

\subsection{Endpoint cap-source control}

We now use the integrable variance bound to control the truncation
residual on regions of bounded cap mean. The resulting integrable
majorant is independent of $k$. Dominated convergence then shows
that the residual vanishes as $k\to\infty$.
For each fixed $M\ge0$, we introduce 
\[
 A_{k,M}=\{(t,x)\in Q_T:t_k(t,x)\le M\}.
\]

\begin{proposition}[Endpoint cap-source control]\label{var:residual}
For every integer $k>2M$,
\begin{equation}\label{var:residual-majorant}
 \mathbf1_{A_{k,M}}\E|a_k(X)|^2\le4V_k\le4V
 \qquad\text{almost everywhere in }Q_T.
\end{equation}
Moreover,
\begin{equation}\label{var:residual-limit}
 \int_{A_{k,M}}\E|a_k(X)|^2\longrightarrow0
 \qquad\text{as }k\to\infty.
\end{equation}
\end{proposition}

\begin{proof}
By the cap bounds \eqref{flux:cap-bounds},
\[
 a_k(s)\le0,\qquad
 |a_k(s)|\le T_k(s)\mathbf1_{\{s>k\}}
             \le s\mathbf1_{\{s>k\}}.
\]
Fix $(t,x)\in A_{k,M}$. If $s>k$, then, since $T_k$ is increasing and
$T_k(k)=k$,
\[
 T_k(s)\ge k>2M\ge2t_k.
\]
Hence
\[
 T_k(s)-t_k\ge \frac12T_k(s),
\]
and therefore
\[
 |a_k(s)|\le T_k(s)\le2|T_k(s)-t_k|.
\]
If $s\le k$, then $T_k(s)=s$ and $T_k'(s)=1$, so $a_k(s)=0$ and the
same inequality holds. Consequently, on $A_{k,M}$,
\[
 \E|a_k(X)|^2
 \le4\E|T_k(X)-t_k|^2
 =4V_k.
\]
Thus
\[
 \mathbf1_{A_{k,M}}\E|a_k(X)|^2\le4V_k\le4V,
\]
where the last inequality follows from Lemma~\ref{var:abstract}. This
proves \eqref{var:residual-majorant}.

Since $V\in L^1(Q_T)$, we have $V(t,x)<\infty$ for almost every
$(t,x)\in Q_T$. At such a point, $\rho(t,x)<\infty$ and
\[
 \E X^2=V+\rho^2<\infty.
\]
For each fixed $s<\infty$, one has $a_k(s)=0$ once $k>s$; hence
$a_k(X)\to0$ pointwise in the Young probability variable. Moreover,
by \eqref{flux:cap-bounds},
\(
 |a_k(X)|^2\le X^2,
\)
and $X^2$ is integrable with respect to $\eta_{t,x}$. Dominated
convergence in the Young probability space therefore gives
\(
 \E|a_k(X)|^2\longrightarrow0
\)
for almost every $(t,x)\in Q_T$.

Let
\[
 F_k(t,x):=\mathbf1_{A_{k,M}}(t,x)\,\E|a_k(X)|^2.
\]
By \eqref{var:residual-majorant},
\[
 0\le F_k\le4V\in L^1(Q_T).
\]
Moreover, for almost every $(t,x)$,
\[
 0\le F_k(t,x)\le\E|a_k(X)|^2\longrightarrow0.
\]
A second application of the dominated convergence theorem, now on
$Q_T$, yields
\[
 \int_{A_{k,M}}\E|a_k(X)|^2
 =\int_{Q_T}F_k\longrightarrow0,
\]
which proves \eqref{var:residual-limit}.
\end{proof}
In the supercritical argument, the cap residual is made small through an
interpolation gain; see \cite[Proposition~7.1, (7.6)]{Feireisl} and
\cite[Lemma~4.4, (4.25)]{FNP}. This gain is no longer available at the
linear pressure law. Proposition~\ref{var:residual} replaces it by the
quadratic oscillation information obtained above: on the region of
bounded cap mean, the residual is dominated by the single integrable
quantity \(4V\). Thus the relevant issue at the endpoint is not the
absence of oscillation control, but how to convert that control into an
integrable domination strong enough to remove the cap residual. No
global \(L^2\)-bound on the density is used in this argument.

\subsection{Cap-mean discrepancy}
We now compare $T_k(\rho)$ with the cap mean $t_k$.
The integrable variance $V$ bounds their squared difference.
Dominated convergence then yields strong $L^2(Q_T)$ convergence
of this difference to zero as $k\to\infty$.

\begin{corollary}[Cap-mean discrepancy]\label{var:discrepancy}
With $e_k:=T_k(\rho)-t_k$,
\begin{equation}\label{var:discrepancy-limit}
 0\le e_k,\qquad e_k\longrightarrow0
                 \quad\text{strongly in }L^2(Q_T).
\end{equation}
Also $t_k\uparrow\rho$ almost everywhere and in $L^1(Q_T)$.
\end{corollary}

\begin{proof}
Since \(T_k(s)\uparrow s\) for every \(s\ge0\), the monotone convergence
theorem in the Young probability variable gives
\[
 t_k=\E T_k(X)\uparrow\E X=\rho
 \qquad\text{a.e. in }Q_T.
\]
Moreover, \(0\le t_k\le\rho\), and \(\rho\in L^1(Q_T)\); hence
\[
 t_k\longrightarrow\rho
 \qquad\text{strongly in }L^1(Q_T).
\]
Since also \(T_k(\rho)\to\rho\) almost everywhere, we have
\[
 e_k=T_k(\rho)-t_k\longrightarrow0
 \qquad\text{a.e. in }Q_T.
\]
By \eqref{var:jensen} and \eqref{var:global},
\[
 0\le e_k^2\le V,\qquad V\in L^1(Q_T).
\]
The dominated convergence theorem therefore yields
\[
 \int_{Q_T}e_k^2\longrightarrow0,
\]
thus  \(e_k\to0\) strongly in \(L^2(Q_T)\).
\end{proof}

\section{Renormalization and truncated entropy closure}\label{sec:closure}

Fix \(Q=(0,T)\times\Omega\), and continue to work with the subsequence and
density Young measure obtained above. The passage from bounded-cap
renormalization to the truncated entropy identity follows the classical
scheme \cite{Feireisl,FNP}. To close the argument at the endpoint, we use
the cap-source estimate of Proposition~\ref{var:residual}. We use \(K\)
for the cap parameter that will be sent to infinity in the
renormalization argument, and reserve \(k\) for the entropy cap fixed
afterwards. The global continuity equation and conservation of mass from
Section~\ref{sec:momentum} remain available throughout.

\subsection{Bounded-cap renormalization}

For fixed \(K\), we return to the artificial-pressure approximation
sequence \((\rho_n,u_n)\) retained throughout the paper, where
\(\delta_n\downarrow0\). Thus \(n\) denotes the artificial-pressure
approximation index. Set
\[
 q_{K,n}=T_K(\rho_n),\qquad t_K=\E T_K(X),\qquad
 a_K(\rho_n)\operatorname{div}u_n\rightharpoonup S_K
       \quad\text{in }L^2(Q).
\]
The source \(S_K\) remains unidentified.

\begin{lemma}[Global cap equation and initial cap trace]\label{close:globalcap}
The zero extensions of $t_K,u,S_K$ satisfy
\begin{equation}\label{close:capequation}
 \partial_t t_K+\operatorname{div}(t_Ku)=-S_K,
 \qquad t_K|_{t=0}=T_K(\rho_0).
\end{equation}
For every smooth scalar test $\psi$ up to the spatial boundary,
vanishing near $t=T$,
\begin{equation}\label{close:capweak}
 \int_Q(t_K\partial_t\psi+t_Ku\cdot\nabla\psi-S_K\psi)
 +\int_\Omega T_K(\rho_0)\psi(0)=0.
\end{equation}
On any ball containing $\overline\Omega$, the zero-extended caps
$q_{K,n}:=T_K(\rho_n)$ converge to $t_K$ in $C([0,T];H^{-1})$.
\end{lemma}
\begin{proof}
We first globalize the approximate cap equation. Let
$d_\partial(x)=\operatorname{dist}(x,\partial\Omega)$ and choose
$\chi_a=0$ for $d_\partial<a$, $\chi_a=1$ for $d_\partial>2a$,
with $|\nabla\chi_a|\le C/d_\partial$. Hardy's inequality gives,
for bounded $w$ and $v\in L^2_tW^{1,2}_0(\Omega)$,
\begin{equation}\label{close:hardy}
 \int_Q|wv\cdot\nabla\chi_a|
 \le C\|w\|_\infty
 |(0,T)\times\{d_\partial<2a\}|^{1/2}
 \|v/d_\partial\|_{L^2(Q)}\longrightarrow0,
\end{equation}
as $a\to 0$. At fixed $K$, this convergence is uniform for
$(w,v)=(q_{K,n},u_n)$. Recall that
\[
 \partial_t q_{K,n}+\operatorname{div}(q_{K,n}u_n)
 =-a_K(\rho_n)\operatorname{div}u_n.
\]
Testing this equation with $\chi_a\psi$, the only boundary-layer term is
\[
 \int_Q q_{K,n}u_n\cdot\nabla\chi_a\,\psi,
\]
which tends to zero uniformly in $n$ by \eqref{close:hardy}. All
remaining terms pass to the limit as $a\to 0$, yielding the global
approximate cap equation with its initial term.

We next establish compactness of the caps. After zero extension to a
fixed ball containing $\overline\Omega$, the sequence $(q_{K,n})_n$ is
bounded in $L^\infty(0,T;L^2)$ and has fixed bounded support. Moreover,
both $q_{K,n}u_n$ and
$a_K(\rho_n)\operatorname{div}u_n$ are uniformly bounded in $L^2(Q)$,
so
\[
 \|\partial_t q_{K,n}\|_{L^2(0,T;H^{-1})}\le C_K.
\]
Consequently, for $0\le s<t\le T$,
\[
 \|q_{K,n}(t)-q_{K,n}(s)\|_{H^{-1}}
 \le C_K|t-s|^{1/2}.
\]
Since $L^2$ is compactly embedded into $H^{-1}$ on the auxiliary ball,
Arzel\`a--Ascoli yields relative compactness in
$C([0,T];H^{-1})$. The weak-$*$ limit of $q_{K,n}=T_K(\rho_n)$ is
$t_K=\E T_K(X)$, hence
\[
 q_{K,n}\longrightarrow t_K
 \qquad\text{in }C([0,T];H^{-1}),\text{ as } n\to \infty.
\]
Pairing this strong cap convergence with the weak
$L^2_tH^1_x$ convergence of $u_n$ gives
$q_{K,n}u_n\to t_Ku$ in distributions as $n\to\infty.$

It remains to identify the initial cap trace. Since
$q_{K,n}(0)=T_K(\rho_{0,n})$, the strong convergence of $q_{K,n}$ in
$C([0,T];H^{-1})$ gives, as $n\to\infty$, we have 
\[
 T_K(\rho_{0,n})\longrightarrow t_K(0)
 \qquad\text{in }H^{-1}.
\]
On the other hand,
\[
 \|T_K(\rho_{0,n})-T_K(\rho_0)\|_{L^p}^p
 \le(4K)^{p-1}\|\rho_{0,n}-\rho_0\|_{L^1}\longrightarrow0
 \quad(1\le p<\infty),
\]
so $t_K(0)=T_K(\rho_0)$ in $H^{-1}$. Passing $n\to\infty$ in the
global approximate cap equation therefore gives \eqref{close:capweak}
and hence \eqref{close:capequation}. The initial trace is thus obtained
without invoking any nonlinear trace for the density Young measure.
\end{proof}


We use the classical renormalization techniques for bounded solutions of
the continuity equation with source; see \cite{DL,Feireisl,FNP}. In the
present low-regularity setting, it follows by spatial mollification and
the standard commutator estimate, which controls the
defect generated by commuting the transport operator with convolution.
Thus, if \(w\in L^\infty\), \(v\in L^2_tH^1_x\), and \(G\in L^2\) satisfy
\[
 \partial_t w+\operatorname{div}(wv)=G,
\]
then, for every \(C^1\) function \(b\) on the bounded range of \(w\),
\[
 \partial_t b(w)+\operatorname{div}(b(w)v)
 +(wb'(w)-b(w))\operatorname{div}v=b'(w)G
\]
in distributions.

Applying this to \eqref{close:capequation} gives
\begin{equation}\label{close:renormcap}
 \partial_t b(t_K)+\operatorname{div}(b(t_K)u)
 +(t_Kb'(t_K)-b(t_K))\operatorname{div}u
 =-b'(t_K)S_K.
\end{equation}
To retain the nonlinear initial value, extend \(t_K\) to \(-1<t<0\)
by \(T_K(\rho_0)\) and set \(u=S_K=0\) there. By
\eqref{close:capweak}, this extension satisfies the cap equation
distributionally across \(t=0\). Renormalizing the bounded extension
therefore yields the initial value \(b(T_K(\rho_0))\). This step uses
neither strong convergence of the density nor renormalization of the
unbounded limiting density.


\subsection{Endpoint source removal}
We now let $K\to\infty$ in the renormalized cap equation.
Using the estimate of Proposition~\ref{var:residual}, we show
that $b'(t_K)S_K\to0$ in $L^1(Q)$.
This yields renormalized continuity for $\rho$, with initial
value $b(\rho_0)$.

\begin{proposition}[Limiting renormalization]\label{close:renormalization}
For every $b\in C^1([0,\infty))$ with compactly supported derivative,
\begin{equation}\label{close:renormrho}
 \partial_tb(\rho)+\operatorname{div}(b(\rho)u)
 +(\rho b'(\rho)-b(\rho))\operatorname{div}u=0
\end{equation}
globally, with initial value $b(\rho_0)$. For the tests of
Lemma~\ref{close:globalcap},
\begin{align}\label{close:renormweak}
 &\int_Q[b(\rho)\partial_t\psi+b(\rho)u\cdot\nabla\psi
       -(\rho b'(\rho)-b(\rho))\operatorname{div}u\,\psi]
       \notag\\
 &\hspace{35mm}+\int_\Omega b(\rho_0)\psi(0)=0.
\end{align}
\end{proposition}
\begin{proof}
Choose $M$ so that $b'=0$ on $[M,\infty)$, and set
$A_{K,M}=\{t_K\le M\}$.

At fixed \(K\), we have
\[
 a_K(\rho_n)\operatorname{div}u_n\rightharpoonup S_K
 \qquad\text{in }L^2(Q).
\]
Hence, for every bounded measurable \(\zeta\),
\[
 \int_Q\zeta S_K
 =
 \lim_{n\to\infty}
 \int_Q\zeta\,a_K(\rho_n)\operatorname{div}u_n.
\]
By the Cauchy--Schwarz inequality,
\[
 \left|
 \int_Q\zeta\,a_K(\rho_n)\operatorname{div}u_n
 \right|
 \le
 \|\operatorname{div}u_n\|_{L^2(Q)}
 \left(
 \int_Q|\zeta|^2|a_K(\rho_n)|^2
 \right)^{1/2}.
\]
Using the definition of \(D_*\), we obtain
\[
 \left|\int_Q\zeta S_K\right|
 \le
 D_*^{1/2}
 \lim_{n\to\infty}
 \left(
 \int_Q|\zeta|^2|a_K(\rho_n)|^2
 \right)^{1/2}.
\]
Since \(a_K^2\) is bounded and continuous, the Young-measure
representation gives
\[
 |a_K(\rho_n)|^2
 \overset{*}{\rightharpoonup}
 \E|a_K(X)|^2
 \qquad\text{in }L^\infty(Q).
\]
Therefore
\[
 \int_Q|\zeta|^2|a_K(\rho_n)|^2
 \longrightarrow
 \int_Q|\zeta|^2\E|a_K(X)|^2,
\]
and hence
\begin{equation}\label{close:sourcebound}
 \left|\int_Q\zeta S_K\right|
 \le
 D_*^{1/2}
 \left(
 \int_Q|\zeta|^2\E|a_K(X)|^2
 \right)^{1/2}.
\end{equation}



Choosing 
\(
\zeta=b'(t_K)\operatorname{sgn}\!\bigl(b'(t_K)S_K\bigr)
\)
in \eqref{close:sourcebound}, then we have 
\[
\int_Q \zeta S_K
=
\int_Q |b'(t_K)S_K|
=
\|b'(t_K)S_K\|_{L^1(Q)}.
\]
Since \(b'=0\) on \([M,\infty)\),
\[
|b'(t_K)|^2
\le
\|b'\|_\infty^2\,\mathbf1_{A_{K,M}},
\;\;\;\text{ and }\;
A_{K,M}=\{t_K\le M\}.
\]
Hence \eqref{close:sourcebound} yields
\[
\begin{aligned}
\|b'(t_K)S_K\|_{L^1(Q)}
&\le
D_*^{1/2}
\left(
\int_Q
|b'(t_K)|^2\,\E|a_K(X)|^2
\right)^{1/2}\\
&\le
\|b'\|_\infty D_*^{1/2}
\left(
\int_{A_{K,M}}\E|a_K(X)|^2
\right)^{1/2}.
\end{aligned}
\]
For \(K>2M\), Proposition~\ref{var:residual} gives
\(
\int_{A_{K,M}}\E|a_K(X)|^2\longrightarrow0,
\)
and therefore
\begin{equation}\label{close:sourcevanishes}
\|b'(t_K)S_K\|_{L^1(Q)}
\longrightarrow0.
\end{equation}


Here we let 
\(
G(s):=sb'(s)-b(s),
\)
note that \(b'\) has compact support, both \(b\) and \(G\) are bounded and
continuous on \([0,\infty)\). As \(t_K\to\rho\) almost everywhere in
\(Q\), the dominated convergence theorem gives
\[
 b(t_K)\to b(\rho),\qquad
 G(t_K)\to G(\rho)
 \quad\text{strongly in }L^2(Q).
\]
Since \(u,\operatorname{div}u\in L^2(Q)\), these convergences allow the
transport and divergence terms in \eqref{close:renormcap} to pass to the
limit. Moreover,
\[
 b(T_K(\rho_0))\to b(\rho_0)
 \quad\text{strongly in }L^1(\Omega)
\]
by bounded dominated convergence. Together with
\eqref{close:sourcevanishes}, this yields \eqref{close:renormweak} as
\(K\to\infty\), provided first that \(b(0)=0\), so that the zero
extension used in the whole-space renormalization argument is compatible
with \(b(t_K)\).

For a general \(b\) with \(b(0)\ne0\), define
\(
 \widetilde b(s):=b(s)-b(0).
\)
Then \(\widetilde b(0)=0\) and \(\widetilde b'=b'\), so the preceding
argument applies to \(\widetilde b\). Restoring the constant \(b(0)\)
does not change the renormalized equation, since
\[
 \operatorname{div}\!\bigl(b(0)u\bigr)
 -b(0)\operatorname{div}u=0.
\]
The initial value is restored simultaneously from
\(\widetilde b(\rho_0)=b(\rho_0)-b(0)\). Hence
\eqref{close:renormweak}, and therefore \eqref{close:renormrho}, hold for
every \(b\in C^1([0,\infty))\) with compactly supported derivative.

\end{proof}

\subsection{Truncated entropy argument}
In this subsection, 
we want to prove strong convergence of the density in $L^1(Q)$.
We first establish the truncated entropy identities and identify
the initial entropy traces. Combining these identities with the
fixed-cap effective-flux identity, we show that the centered
variance $V$ vanishes. Uniform integrability then yields the
strong convergence.

\begin{lemma}[Truncated entropy identities]\label{close:entropyidentity}
For each fixed integer $k\ge1$, define
\begin{equation}\label{close:entropy}
 L_k(s)=s\int_1^s\frac{T_k(z)}{z^2}\,dz\quad(s>0),
 \qquad L_k(0)=0.
\end{equation}
Then $L_k$ is continuous, convex, and of at most linear growth, and
\begin{equation}\label{close:entropyidentities}
 sL_k'(s)-L_k(s)=T_k(s),\qquad
 L_k''(s)=\frac{T_k'(s)}s\ge0\quad(s>0).
\end{equation}
Moreover, the functions
\[
 \ell_{k,n}(t):=\int_\Omega L_k(\rho_n(t)),
 \qquad
 \ell_k(t):=\int_\Omega L_k(\rho(t))
\]
admit absolutely continuous representatives on $[0,T]$ satisfying
\begin{align}
 \ell_{k,n}'(t)&=-\int_\Omega T_k(\rho_n)\operatorname{div}u_n,
 \label{close:entropyapprox}\\
 \ell_k'(t)&=-\int_\Omega T_k(\rho)\operatorname{div}u,
 \label{close:entropylimit}
\end{align}
for almost every $t\in(0,T)$, with initial values
\[
 \ell_{k,n}(0)=\int_\Omega L_k(\rho_{0,n}),
 \qquad
 \ell_k(0)=\int_\Omega L_k(\rho_0).
\]
\end{lemma}

\begin{proof}
A direct differentiation of \eqref{close:entropy} gives
\eqref{close:entropyidentities}, and hence the convexity of $L_k$.
Since $T_k(s)=s$ on $[0,k]$ and $k\ge1$, we have
$L_k(s)=s\log s$ on $(0,1]$; in particular, $L_k$ extends continuously
to zero. Its derivative, however, has the logarithmic singularity at
vacuum, so a regularization will be needed before applying
Proposition~\ref{close:renormalization}.

We first separate the affine tail of $L_k$. Set
\begin{equation}\label{close:affine}
 c_k=\int_1^\infty\frac{T_k(z)}{z^2}\,dz,
 \qquad B_k(s)=L_k(s)-c_ks.
\end{equation}
For $s>0$,
\[
 B_k(s)=-s\int_s^\infty\frac{T_k(z)}{z^2}\,dz.
\]
Thus $B_k$ is bounded. Since $T_k(z)=2k$ for $z\ge3k$,
\[
 B_k(s)=-2k,
 \qquad
 L_k(s)=c_ks-2k
 \qquad\text{for }s\ge3k.
\]
This proves, in particular, that $L_k$ has at most linear growth.

We next remove the logarithmic singularity at zero. For
$0<\eta<1/2$, define
\[
 L_{k,\eta}(s)=L_k(s+\eta)-L_k(\eta),
 \qquad
 B_{k,\eta}(s)=L_{k,\eta}(s)-c_ks.
\]
Then $B_{k,\eta}(0)=0$, and $B_{k,\eta}$ is bounded with continuous,
compactly supported derivative. Hence
Proposition~\ref{close:renormalization} applies to $B_{k,\eta}$.
The renormalization coefficient is unchanged by subtracting the affine
term $c_ks$ and is given by
\[
\begin{aligned}
 g_{k,\eta}(s)
 &:=sL_{k,\eta}'(s)-L_{k,\eta}(s)\\
 &=T_k(s+\eta)-\eta L_k'(s+\eta)+L_k(\eta).
\end{aligned}
\]
The logarithmic behavior of $L_k'$ near zero, the affine tail above,
and the one-Lipschitz property of $T_k$ yield
\begin{equation}\label{close:vacuumestimate}
 \sup_{s\ge0}\bigl(
 |L_{k,\eta}(s)-L_k(s)|+|g_{k,\eta}(s)-T_k(s)|
 \bigr)
 \le C_k\eta(1+|\log\eta|)\longrightarrow0.
\end{equation}
Indeed, the first difference is controlled by integrating $L_k'$ over
an interval of length $\eta$, while the displayed formula for
$g_{k,\eta}$, together with
$|T_k(s+\eta)-T_k(s)|\le\eta$, gives the second.

We now apply Proposition~\ref{close:renormalization} to the bounded
remainder $B_{k,\eta}(\rho)$. Since
\[
 L_{k,\eta}(\rho)=B_{k,\eta}(\rho)+c_k\rho,
\]
adding $c_k$ times the global continuity equation recovers the linear
part and gives
\[
 \partial_tL_{k,\eta}(\rho)
 +\operatorname{div}(L_{k,\eta}(\rho)u)
 +g_{k,\eta}(\rho)\operatorname{div}u=0
\]
globally, with initial value $L_{k,\eta}(\rho_0)$. Thus the bounded
remainder is handled by Proposition~\ref{close:renormalization}, whereas
the linear part $c_k\rho$ is recovered directly from the global
continuity equation; in particular, no boundary estimate on
$\rho u/d_\partial$ is required. The same argument, using the
renormalized continuity equation available for each approximant, gives
\[
 \partial_tL_{k,\eta}(\rho_n)
 +\operatorname{div}(L_{k,\eta}(\rho_n)u_n)
 +g_{k,\eta}(\rho_n)\operatorname{div}u_n=0
\]
with initial value $L_{k,\eta}(\rho_{0,n})$.

Finally, by \eqref{close:vacuumestimate},
$L_{k,\eta}\to L_k$ and $g_{k,\eta}\to T_k$ uniformly on
$[0,\infty)$ as $\eta\downarrow0$, with $k$ fixed.
Since $u_n,u$ and their divergences belong to $L^1(Q)$,
we can pass to the limit in every term of the two weak
formulations. The uniform convergence also gives convergence
of the initial terms. We therefore obtain the global identities
\[
 \partial_tL_k(\rho_n)
 +\operatorname{div}(L_k(\rho_n)u_n)
 +T_k(\rho_n)\operatorname{div}u_n=0,
\]
and
\[
 \partial_tL_k(\rho)
 +\operatorname{div}(L_k(\rho)u)
 +T_k(\rho)\operatorname{div}u=0,
\]
with initial values $L_k(\rho_{0,n})$ and $L_k(\rho_0)$,
respectively.

Testing these identities with smooth functions depending only
on time eliminates the transport terms and yields
\eqref{close:entropyapprox}--\eqref{close:entropylimit}.
Since $0\le T_k\le2k$ and
$\operatorname{div}u_n,\operatorname{div}u\in L^2(Q)$,
the right-hand sides belong to $L^1(0,T)$.
Hence $\ell_{k,n}$ and $\ell_k$ admit absolutely continuous
representatives on $[0,T]$, with the stated initial values.
\end{proof}

\begin{lemma}[Zero initial entropy defect]\label{close:initialdefect}
For each fixed integer $k\ge1$, the sequence
$(\ell_{k,n})_{n\ge1}$ is bounded in $H^1(0,T)$ and hence precompact in
$C([0,T])$. After a diagonal refinement of the approximation sequence
over $k\in\mathbb N$, still denoted by $n$, we may assume that,
simultaneously for every integer $k\ge1$,
\[
 \ell_{k,n}\longrightarrow\overline\ell_k
 \qquad\text{uniformly on }[0,T],
\]
for some $\overline\ell_k\in H^1(0,T)$. Moreover,
\begin{equation}\label{close:barentropy}
 \overline\ell_k(t)
   =\int_\Omega\E L_k(X)(t,x)\,dx
 \quad\text{for a.e. }t,\qquad
 \overline\ell_k'(t)
   =-\int_\Omega h_k(t,x)\,dx
 \quad\text{for a.e. }t,
\end{equation}
and
\begin{equation}\label{close:zerotrace}
 \overline\ell_k(0)
 =\int_\Omega L_k(\rho_0)\,dx
 =\ell_k(0).
\end{equation}
\end{lemma}

\begin{proof}
Fix an integer $k\ge1$. By the linear growth of $L_k$ and the uniform
mass bound,
\[
 \sup_n\|\ell_{k,n}\|_{L^\infty(0,T)}\le C_k.
\]
Moreover, \eqref{close:entropyapprox} and $0\le T_k\le2k$ give
\[
 |\ell_{k,n}'(t)|
 \le
 2k|\Omega|^{1/2}
 \|\operatorname{div}u_n(t)\|_{L^2(\Omega)}.
\]
Hence $(\ell_{k,n})_n$ is bounded in $H^1(0,T)$. Since
\(
 H^1(0,T)\Subset C([0,T]),
\)
for the  fixed $k$ there exists a subsequence $n_j=n_j(k)$ and a
function $\overline\ell_k\in H^1(0,T)$ such that
\[
 \ell_{k,n_j}\longrightarrow\overline\ell_k
 \qquad\text{uniformly on }[0,T].
\]
Since the integer cap levels form a countable family, a diagonal
extraction yields a single subsequence of the approximation index $n$,
still denoted by $n$, for which this convergence holds simultaneously
for every integer $k\ge1$. The diagonal refinement is chosen to retain
all previously established weak limits.

We next identify the scalar limit. Since $L_k$ is continuous and has at
most linear growth,
\(
 |L_k(s)|\le C_k(1+s).
\)
The density sequence $(\rho_n)_n$ is uniformly integrable on $Q$.
Consequently, a standard truncation of $L_k$ by bounded continuous
functions, together with the Young-measure representation, gives
\[
 L_k(\rho_n)
 \rightharpoonup
 \E L_k(X)
 \qquad\text{weakly in }L^1(Q).
\]
Indeed, the Young-measure convergence applies directly to each bounded
truncation of $L_k$, while the linear growth of $L_k$ and the uniform
integrability of $\rho_n$ make the truncation tails uniformly small and
exclude a concentration contribution.

Let $\varphi\in C_c^\infty(0,T)$, then we have
\[
\begin{aligned}
 \int_0^T\varphi(t)\ell_{k,n}(t)\,dt
 &=
 \int_Q\varphi(t)L_k(\rho_n(t,x))\,dx\,dt\\
 &\longrightarrow
 \int_0^T\varphi(t)
 \left(\int_\Omega\E L_k(X)(t,x)\,dx\right)dt.
\end{aligned}
\]
On the other hand, the uniform convergence of $\ell_{k,n}$ gives
\[
 \int_0^T\varphi(t)\ell_{k,n}(t)\,dt
 \longrightarrow
 \int_0^T\varphi(t)\overline\ell_k(t)\,dt.
\]
Thus
\[
 \overline\ell_k(t)
 =
 \int_\Omega\E L_k(X)(t,x)\,dx
 \qquad\text{for a.e. }t.
\]

We now identify the time derivative. By
\eqref{close:entropyapprox},
\[
 \ell_{k,n}'(t)
 =
 -\int_\Omega
 T_k(\rho_n)\operatorname{div}u_n\,dx.
\]
By the fixed-cap convergence \eqref{flux:weak-caps},
\[
 T_k(\rho_n)\operatorname{div}u_n
 \rightharpoonup h_k
 \qquad\text{weakly in }L^2(Q),
\]
we obtain, for every $\varphi\in C_c^\infty(0,T)$,
\[
 \int_0^T\varphi(t)\ell_{k,n}'(t)\,dt
 \longrightarrow
 -\int_0^T\varphi(t)
 \left(\int_\Omega h_k(t,x)\,dx\right)dt.
\]
At the same time, the $H^1(0,T)$ boundedness and the uniform convergence
of $\ell_{k,n}$ imply
\[
 \ell_{k,n}'\rightharpoonup\overline\ell_k'
 \qquad\text{weakly in }L^2(0,T).
\]
Uniqueness of the weak limit therefore gives
\[
 \overline\ell_k'(t)
 =
 -\int_\Omega h_k(t,x)\,dx
 \qquad\text{for a.e. }t,
\]
which completes the proof of \eqref{close:barentropy}.

It remains to identify the initial value. Since
\[
 \rho_{0,n}\longrightarrow\rho_0
 \qquad\text{strongly in }L^1(\Omega),
\]
the sequence $(\rho_{0,n})_n$ is uniformly integrable and converges to
$\rho_0$ in measure. The continuity and linear growth of $L_k$ therefore
give, by Vitali's theorem,
\[
 L_k(\rho_{0,n})
 \longrightarrow
 L_k(\rho_0)
 \qquad\text{strongly in }L^1(\Omega).
\]
Consequently,
\[
 \ell_{k,n}(0)
 =\int_\Omega L_k(\rho_{0,n})\,dx
 \longrightarrow
 \int_\Omega L_k(\rho_0)\,dx.
\]
The uniform convergence on the closed interval $[0,T]$ also gives
\(
 \ell_{k,n}(0)\longrightarrow\overline\ell_k(0),
\)
and hence
\(
 \overline\ell_k(0)
 =
 \int_\Omega L_k(\rho_0)\,dx.
\)
The second equality in \eqref{close:zerotrace} follows independently
from the initial value in \eqref{close:entropylimit}. This proves
\eqref{close:zerotrace}.
\end{proof}

The uniform scalar compactness above identifies the nonlinear entropy
trace at $t=0$; the weak initial trace of the capped density alone is
insufficient for this purpose.


\begin{proposition}[Strong density compactness]\label{close:strongdensity}
Along the common subsequence,
\begin{equation}\label{close:strongLone}
 \rho_n\longrightarrow\rho\qquad\text{strongly in }L^1(Q).
\end{equation}
\end{proposition}
\begin{proof}
Let
\[
 D_k(t):=\overline\ell_k(t)-\ell_k(t).
\]
By Lemma~\ref{close:initialdefect}, the definition of $\ell_k$, and
Jensen's inequality for the convex function $L_k$,
\[
 D_k(t)
 =\int_\Omega\bigl(\E L_k(X)-L_k(\rho)\bigr)\,dx
 \ge0
 \qquad\text{for a.e. }t,
\]
since $\E X=\rho$. Both $\overline\ell_k$ and $\ell_k$ belong to
$H^1(0,T)$, hence $D_k$ is continuous on $[0,T]$. Its a.e.
nonnegativity therefore implies $D_k(t)\ge0$ for every $t\in[0,T]$,
and Lemma~\ref{close:initialdefect} gives
\(
 D_k(0)=0.
\)

We next derive the evolution of the defect. By
\eqref{close:barentropy} and \eqref{close:entropylimit}, for almost every
$t\in(0,T)$,
\[
\begin{aligned}
 D_k'(t)
 &=\overline\ell_k'(t)-\ell_k'(t)\\
 &=-\int_\Omega h_k\,dx
   +\int_\Omega T_k(\rho)\operatorname{div}u\,dx.
\end{aligned}
\]
The fixed-cap flux identity \eqref{flux:identity} reads
\[
 C_k=\nu\bigl(h_k-t_k\operatorname{div}u\bigr),
\]
and therefore
\[
 h_k=t_k\operatorname{div}u+\frac1\nu C_k.
\]
Substituting this expression for $h_k$ gives
\[
\begin{aligned}
 D_k'(t)
 &=-\int_\Omega
    \left(t_k\operatorname{div}u+\frac1\nu C_k\right)dx
   +\int_\Omega T_k(\rho)\operatorname{div}u\,dx\\
 &=-\frac1\nu\int_\Omega C_k\,dx
   +\int_\Omega\bigl(T_k(\rho)-t_k\bigr)
                  \operatorname{div}u\,dx.
\end{aligned}
\]
Thus, with
\(
 e_k:=T_k(\rho)-t_k,
\)
we have
\begin{equation}\label{close:defect}
 D_k'(t)
 =-\frac1\nu\int_\Omega C_k\,dx
  +\int_\Omega e_k\operatorname{div}u\,dx
\end{equation}
for almost every $t\in(0,T)$.

Integrating \eqref{close:defect} from $0$ to $t$ and using $D_k(0)=0$
yields
\[
 D_k(t)+\frac1\nu\int_0^t\!\int_\Omega C_k
 =\int_0^t\!\int_\Omega e_k\operatorname{div}u.
\]
Since $D_k(t)\ge0$ and $C_k\ge0$ by
Lemma~\ref{var:abstract}, the left-hand side is nonnegative. Hence,
by the Cauchy--Schwarz inequality,
\begin{equation}\label{close:defectintegrated}
\begin{aligned}
 0
 &\le D_k(t)+\frac1\nu\int_0^t\!\int_\Omega C_k=\int_0^t\!\int_\Omega e_k\operatorname{div}u\le
 \|e_k\|_{L^2(Q)}\|\operatorname{div}u\|_{L^2(Q)}.
\end{aligned}
\end{equation}
By \eqref{var:discrepancy-limit},
$e_k\to0$ strongly in $L^2(Q)$, so the right-hand side of
\eqref{close:defectintegrated} tends to zero as $k\to\infty$.
Taking $t=T$ therefore gives
\(
 \int_Q C_k\longrightarrow0.
\)
On the other hand, $C_k\uparrow V$ almost everywhere by
\eqref{var:monotone}. The monotone convergence theorem yields
\[
 \int_Q V
 =\lim_{k\to\infty}\int_Q C_k
 =0.
\]
Since $V\ge0$,
\begin{equation}\label{close:zerovariance}
 V=0\qquad\text{a.e. on }Q.
\end{equation}
Consequently,
\[
 \E|X-\rho|^2=0
\]
for almost every $(t,x)$, and hence the density Young measure is the
Dirac mass
\[
 \eta_{t,x}=\delta_{\rho(t,x)}
 \qquad\text{for a.e. }(t,x)\in Q.
\]

We now show that this implies convergence in measure. For
$\varepsilon>0$, define the bounded Carath\'eodory function
\[
 \Phi_\varepsilon(t,x,z)
 :=\min\left\{1,
       \frac{|z-\rho(t,x)|^2}{\varepsilon^2}\right\}.
\]
The Young-measure representation gives
\[
 \int_Q\Phi_\varepsilon(t,x,\rho_n(t,x))\,dx\,dt
 \longrightarrow
 \int_Q\left\langle\eta_{t,x},
        \Phi_\varepsilon(t,x,\cdot)\right\rangle dx\,dt=0,
\]
because $\eta_{t,x}=\delta_{\rho(t,x)}$. Since
$\Phi_\varepsilon(t,x,\rho_n)=1$ on
$\{|\rho_n-\rho|\ge\varepsilon\}$, it follows that
\[
 \bigl|\{(t,x)\in Q:|\rho_n-\rho|\ge\varepsilon\}\bigr|
 \longrightarrow0,
\]
which implies that 
\[
 \rho_n\longrightarrow\rho
 \qquad\text{in measure on }Q.
\]

Now we need  to upgrade this convergence to $L^1(Q)$. The entropy part of
\eqref{app:uniformenergy} gives
\[
 \sup_n\operatorname*{ess\,sup}_{0<t<T}
 \int_\Omega \rho_n\log(1+\rho_n)\,dx\le C.
\]
Note that on the set $\{\rho_n>R\}$, we have 
$\log(1+\rho_n)\ge\log(1+R)$, and therefore
\begin{equation}\label{close:UI}
\begin{aligned}
 \sup_n\int_{\{(t,x)\in Q:\,\rho_n>R\}}\rho_n\,dx\,dt
 &\le
 \frac{1}{\log(1+R)}
 \sup_n\int_Q\rho_n\log(1+\rho_n)\,dx\,dt\\
 &\le \frac{CT}{\log(1+R)}
 \longrightarrow0
 \qquad(R\to\infty).
\end{aligned}
\end{equation}
Hence $(\rho_n)_n$ is uniformly integrable on $Q$.

The entropy-tail estimate above implies uniform integrability of
\((\rho_n)_n\) on \(Q\). Indeed, for every measurable set
\(E\subset Q\) and every \(R>0\),
\[
 \int_E \rho_n
 \le
 R|E|
 +
 \int_{\{\rho_n>R\}}\rho_n .
\]
Given \(\varepsilon>0\), first choose \(R\) so large that the second term
is smaller than \(\varepsilon/2\), uniformly in \(n\), and then choose
\(|E|<\varepsilon/(2R)\). Thus
\(
 \sup_n\int_E\rho_n<\varepsilon
\)
whenever \(|E|\) is sufficiently small. Hence \((\rho_n)_n\) is
uniformly integrable on \(Q\). Together with convergence in measure,
Vitali's theorem gives
\[
 \rho_n\longrightarrow\rho
 \qquad\text{strongly in }L^1(Q),
\]
which is \eqref{close:strongLone}.

\end{proof}

The conclusion is global although the earlier density gain and
momentum compactness were local. The global variance bound,
the entropy identities integrated over $\Omega$, and uniform
integrability of the density yield strong density convergence
on the whole cylinder $Q$. Neither uniform integrability of
the entropy densities nor global disappearance of artificial
pressure is required.

\section{Energy admissibility and global existence}\label{sec:existence}

With strong density convergence and renormalized continuity established
in Section~\ref{sec:closure}, we turn to the proofs of the main results.
We first complete the proof of Theorem~\ref{thm:stability} and then
deduce Theorem~\ref{thm:main} from the approximation constructed in
Section~\ref{sec:approx}.

\begin{proof}[Proof of Theorem~\ref{thm:stability}]
Section~\ref{sec:momentum} gives momentum and convection compactness,
the distributional equations for spatially interior momentum tests,
and the initial traces. Sections~\ref{sec:flux}--\ref{sec:closure}
give global renormalization and strong density convergence.
It remains to establish the initial-time energy inequality and to
extend the momentum equation to the zero-boundary test class in
Definition~\ref{def:renormalized-solution}.

We first pass to the limit in the energy inequality. Consider the
nonnegative lower semicontinuous convex energy density
\begin{equation}\label{exist:perspective}
\begin{aligned}
 e(r,z)&=\frac{|z|^2}{2r}+H(r),\qquad r>0,\\
 e(0,0)&=1,\qquad e(0,z)=+\infty\quad(z\ne0).
\end{aligned}
\end{equation}
The value at $(0,0)$ agrees with the normalization $H(0)=1$.
To use the available space-time convergence, we first integrate the
approximate energy inequality against an arbitrary nonnegative time
weight. This avoids requiring convergence of the approximate energies
at individual times.

Let $\theta\in C_c^\infty(0,T)$, $\theta\ge0$, and set
\[
 \Theta(s)=\int_s^T\theta(t)\,dt.
\]
Multiplying \eqref{exist:approxenergy} by $\theta(t)$ and integrating
in $t$, Fubini's theorem rewrites the accumulated dissipation with
weight $\Theta$. Dropping the nonnegative artificial internal energy
then gives
\begin{equation}\label{exist:weighted}
 \int_{Q_T}\theta(t)e(\rho_n,m_n)
 +\int_{Q_T}\Theta(t)\mathbb S(u_n):\nabla u_n
 \le E_{0,n}\int_0^T\theta(t)\,dt.
\end{equation}
By Proposition~\ref{close:strongdensity} and \eqref{mom:weak},
\[
 \rho_n\longrightarrow\rho
 \quad\text{strongly in }L^1(Q_T),
 \qquad
 m_n\rightharpoonup m
 \quad\text{weakly in }L^1(Q_T).
\]
Convex lower semicontinuity therefore yields
\[
 \int_{Q_T}\theta(t)e(\rho,m)
 \le
 \liminf_{n\to\infty}
 \int_{Q_T}\theta(t)e(\rho_n,m_n).
\]
For the dissipation, the 2D identity
\[
 \mathbb S(v):\nabla v
 =2\mu|D^0v|^2+(\mu+\lambda)|\operatorname{div}v|^2,
 \qquad
 D^0v=\frac{\nabla v+\nabla v^{\mathsf T}}2
       -\frac12\operatorname{div}v\,I
\]
expresses the integrand as a nonnegative convex quadratic form.
Since $\sqrt{\Theta}$ is bounded, weak $L^2$ convergence of
$\nabla u_n$ gives lower semicontinuity of the weighted dissipation
as well. Using $E_{0,n}\to E_0$ in \eqref{exist:weighted}, we obtain
\[
 \int_{Q_T}\theta(t)e(\rho,m)
 +\int_{Q_T}\Theta(t)\mathbb S(u):\nabla u
 \le E_0\int_0^T\theta(t)\,dt.
\]
Applying Fubini's theorem once more gives
\[
 \int_0^T\theta(t)\left[
 \int_\Omega e(\rho,m)(t)\,dx
 +\int_0^t\!\int_\Omega\mathbb S(u):\nabla u\,dx\,ds
 -E_0\right]dt\le0.
\]
Since this holds for every nonnegative
$\theta\in C_c^\infty(0,T)$, the bracketed quantity is nonpositive
for almost every $t\in(0,T)$. Together with $m=\rho u$, this proves
\eqref{eq:energy}.

We next extend the momentum test class. The limit equation obtained
in Section~\ref{sec:momentum} is already valid for tests supported
strictly inside $\Omega$. Let $\varphi$ be a smooth vector test on
$[0,T)\times\overline\Omega$ with
$\varphi|_{\partial\Omega}=0$. Set
\[
 d_\partial(x)=\operatorname{dist}(x,\partial\Omega),
 \qquad
 \Omega_a=\{x\in\Omega:d_\partial(x)<2a\}.
\]
Choose smooth cutoffs $\xi_a\in C_c^\infty(\Omega)$ satisfying
\[
 0\le\xi_a\le1,\qquad
 \xi_a=0\ \text{if }d_\partial<a,\qquad
 \xi_a=1\ \text{if }d_\partial>2a,
 \qquad
 |\nabla\xi_a|\le \frac{C}{d_\partial}.
\]
Such cutoffs can be obtained by smoothing distance cutoffs.
Since $\varphi$ vanishes on the boundary and has bounded spatial
derivatives,
\[
 |\varphi(t,x)|\le C d_\partial(x),
 \qquad
 |\varphi\otimes\nabla\xi_a|\le C,
\]
with constants independent of $a$.

For each fixed $a>0$, the field $\xi_a\varphi$ is an admissible
interior test. Expanding
\[
 \nabla(\xi_a\varphi)
 =\xi_a\nabla\varphi+\varphi\otimes\nabla\xi_a,
\]
we see that the additional spatial derivative terms are
\[
 R_a=
 \int_{Q_T}
 \bigl(\rho u\otimes u+\rho I-\mathbb S(u)\bigr)
 :(\varphi\otimes\nabla\xi_a).
\]
The preceding bounds and the support of $\nabla\xi_a$ imply
\[
 |R_a|
 \le
 C\int_{(0,T)\times\Omega_a}
 \bigl(\rho|u|^2+\rho+|\mathbb S(u)|\bigr)\,dx\,dt.
\]
The limiting energy bounds give
$\rho|u|^2,\rho\in L^1(Q_T)$, while
$u\in L^2(0,T;W^{1,2}_0(\Omega))$ gives
$\mathbb S(u)\in L^2(Q_T)\subset L^1(Q_T)$.
Thus the integrand is a fixed $L^1(Q_T)$ function, and absolute
continuity of its integral over the shrinking boundary layer yields
$R_a\to0$.


We now pass to the limit in the remaining terms. Since $\xi_a$
is independent of time,
\(
 \partial_t(\xi_a\varphi)=\xi_a\partial_t\varphi.
\)
Thus the momentum identity tested with $\xi_a\varphi$ reads
\begin{align*}
 0={}&
 \int_{Q_T}\xi_a
 \bigl[
 m\cdot\partial_t\varphi
 +(\rho u\otimes u):\nabla\varphi
 +\rho\operatorname{div}\varphi
 -\mathbb S(u):\nabla\varphi
 \bigr]\,dx\,dt\\
 &+\int_\Omega\xi_a m_0\cdot\varphi(0)\,dx
 +R_a.
\end{align*}
Here $R_a$ contains all terms in which a spatial derivative falls
on $\xi_a$.

Because $\xi_a=1$ outside $\Omega_a$ and $0\le\xi_a\le1$, replacing
$\xi_a$ by $1$ in the space-time integral produces an error bounded by
\[
 C_\varphi
 \int_{(0,T)\times\Omega_a}
 \bigl(
 |m|+\rho|u|^2+\rho+|\mathbb S(u)|
 \bigr)\,dx\,dt,
\]
where $C_\varphi$ depends only on the bounded derivatives of the
fixed test field $\varphi$. All four terms in the integrand belong
to $L^1(Q_T)$. Since the boundary layers $\Omega_a$ shrink to the
empty set inside $\Omega$ as $a\downarrow0$, absolute continuity
of the integral shows that this error tends to zero. This
justifies the passage in both the time term and the spatial
terms containing $\xi_a$.

The initial term is treated in the same way:
\[
 \left|
 \int_\Omega(\xi_a-1)m_0\cdot\varphi(0)\,dx
 \right|
 \le
 \|\varphi(0)\|_{L^\infty(\Omega)}
 \int_{\Omega_a}|m_0|\,dx
 \longrightarrow0,
\]
because $m_0\in L^1(\Omega)$.

The terms involving $\nabla\xi_a$ require the zero boundary trace
of $\varphi$. Indeed, this trace gives
$|\varphi|\le C d_\partial$, which compensates for the bound
$|\nabla\xi_a|\le C/d_\partial$. Consequently, the estimate for
$R_a$ established above shows that these terms also tend to zero.
No boundary traces of the convection or viscous stress are needed;
the argument uses their volume integrability and the vanishing
boundary trace of the test field.

Letting $a\downarrow0$ in the tested identity therefore gives
\[
 \int_{Q_T}
 \bigl[
 m\cdot\partial_t\varphi
 +(\rho u\otimes u):\nabla\varphi
 +\rho\operatorname{div}\varphi
 -\mathbb S(u):\nabla\varphi
 \bigr]\,dx\,dt
 +\int_\Omega m_0\cdot\varphi(0)\,dx=0.
\]
This is \eqref{weak:mom} for every smooth zero-boundary test
$\varphi$ in the stated class. The boundary passage is performed
on the fixed limiting fields, after $n\to\infty$, and therefore
requires no global artificial-pressure convergence.

\end{proof}

Theorem~\ref{thm:main} follows by applying the stability theorem to
the artificial-pressure solutions constructed in
Section~\ref{sec:approx}. A diagonal extraction ensures that the
same subsequence converges on every finite time interval, yielding
a global solution with the prescribed initial data.

\begin{proof}[Proof of Theorem~\ref{thm:main}]
If $\int_\Omega\rho_0\,dx=0$, then $\rho_0=m_0=0$ almost
everywhere, and $(\rho,u)=(0,0)$ is a global solution with
$E_0=|\Omega|$. The approximation assertion follows by taking
$\rho_n=1/n$, $u_n=0$, and $\delta_n=1/n$: these are
artificial-pressure solutions with spatially constant pressure,
their initial augmented energies converge to $|\Omega|$, and
$\rho_n\to0$ strongly in $L^1(Q_T)$ for every finite $T$.
We henceforth assume that the initial mass is positive.

\medskip
\noindent
\emph{Step 1: Construction of the approximating sequence.}
Apply Lemma~\ref{app:data} to obtain smooth initial data satisfying
\[
 \rho_{0,n}\longrightarrow\rho_0,\qquad
 m_{0,n}\longrightarrow m_0
 \quad\text{in }L^1(\Omega),
\]
together with convergence of the natural initial energies.
Choose $\delta_n$ as in \eqref{app:delta}. Then
\[
 \delta_n\int_\Omega\rho_{0,n}^{\beta}\,dx
 \le\frac1n\longrightarrow0,
\]
so the initial augmented energies satisfy
$E_{0,n}\to E_0$ by \eqref{app:initialenergy}.

For each $n$, choose a global solution $(\rho_n,u_n)$ furnished
by Proposition~\ref{app:fixed}, with initial data
$(\rho_{0,n},m_{0,n})$ and pressure
$\rho_n+\delta_n\rho_n^\beta$.
Fix this choice for the remainder of the proof.
The individual energy inequalities, the boundedness of $E_{0,n}$,
and the coercivity identity \eqref{app:coercivity} give
\eqref{app:uniformenergy} on every finite time interval.
Lemma~\ref{app:PW}, applied to these same solutions with
$\alpha=1/8$, supplies the local density gain required by
Theorem~\ref{thm:stability}.
Thus this single sequence satisfies all hypotheses of that theorem
on every finite time interval.

\medskip
\noindent
\emph{Step 2: Passage to a global limit.}
Applying Theorem~\ref{thm:stability} successively on
$Q_L=(0,L)\times\Omega$, $L=1,2,\ldots$, and using a standard
diagonal extraction, we obtain a single subsequence indexed by
$n_k$ and a pair $(\rho,u)$ such that, for every $T<\infty$,
\[
 \rho_{n_k}\longrightarrow\rho
 \quad\text{strongly in }L^1(Q_T),
 \qquad
 u_{n_k}\rightharpoonup u
 \quad\text{weakly in }L^2(0,T;W^{1,2}_0(\Omega)).
\]
The extractions are nested, so the limits agree
on overlapping intervals, including their weakly continuous
density and momentum representatives.
By Theorem~\ref{thm:stability}, the restriction of $(\rho,u)$
to $(0,T)$ for every $T>0$ is a renormalized finite-energy weak
solution with initial data $(\rho_0,m_0)$.
Thus $(\rho,u)$ is a global solution in the sense of
Definition~\ref{def:renormalized-solution}, and the displayed
strong density convergence gives the asserted approximation.

\end{proof}

\appendix
\section{Approximation details and uniform PW constants}\label{appendix:approx}

This appendix supplies the technical inputs of Section~\ref{sec:approx}.
The fixed artificial-pressure limits are taken before the endpoint
compactness argument, with constants allowed to depend on $\delta$.

\subsection{Energy-consistent initial regularization}\label{appendix:initial}

\begin{proof}[Proof of Lemma~\ref{app:data}]
Set $u_0=m_0/\rho_0$ on $\{\rho_0>0\}$ and zero elsewhere.
Finite energy gives $M_0=\int\rho_0<\infty$ and
$u_0\in L^2(\rho_0\,dx)$. Density of $C_c^\infty(\Omega;\mathbb R^2)$
in this weighted space provides $v_n$ with
$\|\sqrt{\rho_0}(v_n-u_0)\|_2\le n^{-1}$.

For $R\ge1$, $s_R=\rho_0\wedge R$ satisfies $H(s_R)\le H(\rho_0)$,
so $(s_R,H(s_R))\to(\rho_0,H(\rho_0))$ in $L^1$ by dominated
convergence. With a nonnegative mollifier and zero extension, put
$r_{R,h,c}=\varphi_h*\widetilde s_R+c$. At fixed $R$, these smooth
positive densities lie in $[c,R+c]$ and converge to $s_R$ in $L^1$
and almost everywhere as $h,c\downarrow0$. Continuity of $H$ on
$[0,R+1]$ gives the same $L^1$ convergence for their energies.
A diagonal choice therefore produces smooth $r_j$, bounded above and
away from zero individually, with
$(r_j,H(r_j))\to(\rho_0,H(\rho_0))$ in $L^1$.

Choose $j(n)$ increasing so that
\[
 (1+\|v_n\|_\infty^2)\|r_{j(n)}-\rho_0\|_1\le n^{-1},\qquad
 \|H(r_{j(n)})-H(\rho_0)\|_1\le n^{-1},
\]
and set $\rho_{0,n}=r_{j(n)}$, $u_{0,n}=v_n$.
The inequality $|\sqrt a-\sqrt b|^2\le|a-b|$ yields
\[
 \left\|\sqrt{\rho_{0,n}}u_{0,n}
        -\frac{m_0}{\sqrt{\rho_0}}\right\|_2
 \le n^{-1}+\|v_n\|_\infty\|\rho_{0,n}-\rho_0\|_1^{1/2}\to0.
\]
Moreover,
\[
 \|\rho_{0,n}v_n-m_0\|_1
 \le\|v_n\|_\infty\|\rho_{0,n}-\rho_0\|_1
      +\sqrt{M_0}\|\sqrt{\rho_0}(v_n-u_0)\|_2\to0.
\]
These are the claimed convergences, including the exact initial
kinetic and internal energies.
\end{proof}

\subsection{The fixed-pressure construction on a \texorpdfstring{$C^2$}{C2} domain}
\label{appendix:fixed}

\begin{proof}[Proof of Proposition~\ref{app:fixed}]
We use the approximation and effective-flux framework of
\cite[Sections~2--3]{FNP}, with the logarithmic physical potential and
planar exponents specified below. The source treats three-dimensional
$C^{2+\eta}$ domains; the planar exponents, logarithmic potential,
$C^2$ parabolic input and initial-energy passages are supplied here. At fixed
$\delta$, the pressure is increasing and
$\lim_{z\to\infty}p_\delta(z)z^{-\beta}=\delta>0$; thus the general-pressure
extension in \cite[Section~7.5, p.~205, and Example~7.1, p.~207]{Lions} also corroborates the
high-power regime. Its smooth-domain convention is supplemented here
by finite-exponent $C^2$ Neumann regularity.

\emph{Galerkin approximation and energy.}
Choose $0<c_*\le\inf r_0$ and $\chi_\varepsilon\in C_c^\infty(\Omega)$,
$0\le\chi_\varepsilon\le1$, tending pointwise to one. The data
$r_{0,\varepsilon}=c_*+\chi_\varepsilon(r_0-c_*)$ are constant near the
boundary and converge to $r_0$ in every finite $L^q$ and in energy.
Let $X_N$ be nested finite-dimensional spaces of compactly supported
smooth vector fields, containing $v_0$ and dense in every finite
$W^{1,q}_0$. Solve
\begin{align}
 \partial_t r_N+\operatorname{div}(r_Nv_N)
 &=\varepsilon\Delta r_N,\qquad
 \partial_{\boldsymbol n}r_N=0,\quad r_N(0)=r_{0,\varepsilon},
 \label{app:parabolic}\\
 \frac{d}{dt}\int_\Omega r_Nv_N\cdot w
 &=\int_\Omega\left(r_Nv_N\otimes v_N:\nabla w
                  +p_\delta(r_N)\operatorname{div}w
                  -\mathbb S(v_N):\nabla w\right)\notag\\
 &\quad-\varepsilon\int_\Omega(\nabla r_N\cdot\nabla)v_N\cdot w,
 \qquad w\in X_N,\quad v_N(0)=v_0.\label{app:Galerkin}
\end{align}
Local solvability uses finite-exponent Neumann regularity;
for the scalar Laplacian on $C^2$ domains this follows from
Denk--Hieber--Pr\"uss, Theorem~2.1.\footnote{R. Denk, M. Hieber and
J. Pr\"uss, \emph{Optimal $L_p$--$L_q$-regularity for parabolic problems
with inhomogeneous boundary data}, Konstanzer Schriften in Mathematik
und Informatik, no.~205 (2005), Theorem~2.1, p.~7, and Proposition~3.3,
p.~9:
\href{https://d-nb.info/1093431571/34}{original preprint}.}
The scalar order-two operator is normally elliptic, the normal Neumann
operator satisfies the complementing condition, and the smooth
boundary-compatible data satisfy its trace conditions.
Comparison keeps $r_N$ positive. With
$\widehat H_\delta(r)=r\log r+\delta r^\beta/(\beta-1)$,
$r\widehat H_\delta'-\widehat H_\delta=p_\delta$ and
$\widehat H_\delta''=r^{-1}+\delta\beta r^{\beta-2}$.
The Galerkin diffusion correction cancels the kinetic cross term,
giving
\begin{equation}\label{app:exactenergy}
 \frac{d}{dt}\widehat E_N
 +\int_\Omega\mathbb S(v_N):\nabla v_N
 +\varepsilon\int_\Omega
       (r_N^{-1}+\delta\beta r_N^{\beta-2})|\nabla r_N|^2=0,
\end{equation}
where $\widehat E_N=\int(r_N|v_N|^2/2+\widehat H_\delta(r_N))$.
Coercivity \eqref{app:coercivity} and finite-dimensional norm
 equivalence control $\int\|\operatorname{div}v_N\|_\infty$ on each finite
time interval; comparison keeps the mass matrix positive and energy controls
its velocity coefficients. The approximate solution therefore continues.

\emph{The limit $N\to\infty$ at fixed $\varepsilon,\delta$.}
Energy gives $r_N\in L^\infty_tL^\beta_x$, $v_N\in L^2_tH^1_0$,
$\nabla r_N\in L^2(Q_T)$ and
$r_N^{\beta/2}\in L^\infty_tL^2_x\cap L^2_tH^1_x$, hence
$r_N\in L^{2\beta}(Q_T)$ by the planar space-time estimate.
For $a=1/2-1/\beta$, spatial H\"older yields
\[
 \|r_Nv_N\|_2
 \le\|r_N|v_N|^2\|_1^a\|r_N\|_\beta^{1-a}
       \|v_N\|_{4\beta/(\beta-2)}^{2/\beta}.
\]
Thus $r_Nv_N\in L^\beta_tL^2_x$. The continuity equation gives
$\partial_t r_N\in L^2_tH^{-1}_x$, and compactness implies
\begin{equation}\label{app:firstdensity}
 r_N\to r_\varepsilon\text{ in }L^q(Q_T),\ q<2\beta,
 \qquad p_\delta(r_N)\to p_\delta(r_\varepsilon)\text{ in }L^1(Q_T).
\end{equation}
For the negative-space heat estimate, put $X=(H^1(\Omega))^{\prime}$
and define $A_\varepsilon:H^1(\Omega)\to X$ by
\[
 \langle A_\varepsilon w,\phi\rangle
 =\int_\Omega(w\phi+\varepsilon\nabla w\cdot\nabla\phi),
 \qquad \phi\in H^1(\Omega).
\]
As an operator in $X$, $A_\varepsilon=I-\varepsilon\Delta_N$ has domain
$H^1(\Omega)$ and is positive self-adjoint in the equivalent Hilbert
norm $\|f\|_X=\|(I-\Delta_N)^{-1/2}f\|_2$, defined spectrally by the
Neumann Laplacian. It is therefore invertible and $R$-sectorial of
angle zero. Regard \eqref{app:parabolic} as
\[
 \partial_t r_N+A_\varepsilon r_N
 =r_N-\operatorname{div}(r_Nv_N)\in L^\beta(0,T;X).
\]
Here divergence is defined by pairing with $-\nabla\phi$; the forcing
bound follows from $r_Nv_N\in L^\beta_tL^2_x$ and
$r_N\in L^\infty_tL^\beta_x$. The smooth datum satisfies
$r_{0,\varepsilon}\in H^1(\Omega)\subset
(X,H^1(\Omega))_{1-1/\beta,\beta}$.
Denk--Hieber--Pr\"uss, Proposition~3.3, p.~9 of the same preprint,
applied with time exponent $\beta$ (extending the forcing by zero
past $T$), yields
\[
 \begin{aligned}
 &\|\partial_t r_N\|_{L^\beta(0,T;X)}
   +\|r_N\|_{L^\beta(0,T;H^1)}\\
 &\quad\le C_{\varepsilon,\beta,T,\Omega}\bigl(
 \|r_Nv_N\|_{L^\beta_tL^2_x}+\|r_N\|_{L^\beta_tX}
 +\|r_{0,\varepsilon}\|_{(X,H^1)_{1-1/\beta,\beta}}\bigr).
 \end{aligned}
\]
Thus $\nabla r_N\in L^\beta_tL^2_x$, with a bound depending on
$\varepsilon,\delta,\beta,T,\Omega$ and the prescribed data,
independently of $N$. Consequently,
$-v_N\cdot\nabla r_N-r_N\operatorname{div}v_N\in L^{4/3}(Q_T)$;
here use $v_N\in L^2_tL^4_x$, $r_N\in L^{2\beta}(Q_T)$ and
$\beta\ge6$. The preceding scalar $C^2$ theorem at exponent $4/3$
then gives $\partial_t r_N,\Delta r_N\in L^{4/3}(Q_T)$.

Momentum ranges are uniformly bounded in
$L^{2\beta/(\beta+1)}\Subset H^{-1}$. For each fixed Galerkin test,
\eqref{app:Galerkin} bounds its scalar time derivative in
$L^{2\beta/(\beta+2)}$, whose exponent exceeds one.
Coordinate Ascoli and compactness of the ranges therefore give
momentum convergence in $C_tH^{-1}$; no norm bound for Galerkin
projections is required. Strong--weak pairing identifies convection.
The density-square balance
\[
 \frac12\frac{d}{dt}\|r_N\|_2^2+
 \varepsilon\|\nabla r_N\|_2^2
 =-\frac12\int r_N^2\operatorname{div}v_N
\]
and $r_N^2\to r_\varepsilon^2$ in $L^2(Q_T)$ give convergence of the
weighted gradient norms. Hence $\nabla r_N\to\nabla r_\varepsilon$
strongly on $Q_{T'}$, $T'<T$, identifying the diffusion correction.
The construction on a larger time interval covers the prescribed $Q_T$.

To retain the initial energy, integrate \eqref{app:exactenergy} with
$\theta\ge0$ and $\Theta(s)=\int_s^T\theta(t)dt$:
\begin{equation}\label{app:weightedinner}
 \int_0^T\theta\widehat E_N
 +\int_0^T\Theta(D_N+\varepsilon A_N)
 =\widehat E_{0,\varepsilon}\int_0^T\theta,
\end{equation}
where $D_N=\int\mathbb S(v_N):\nabla v_N$ and
$A_N=4\int|\nabla\sqrt{r_N}|^2+(4\delta/\beta)\int|\nabla r_N^{\beta/2}|^2$.
Convex and quadratic lower semicontinuity gives the parabolic energy
inequality with the same initial right-hand side.

\emph{Removal of density diffusion at fixed $\delta$.}
The Bogovski\u\i\ test with
$\mathcal B[r_\varepsilon-\langle r_\varepsilon\rangle]$,
as in \cite[Lemma~3.1]{FNP}, gives
\begin{equation}\label{app:fixedgain}
 \int_{Q_T}p_\delta(r_\varepsilon)r_\varepsilon\le C_\delta,
 \qquad\sqrt\varepsilon\|\nabla r_\varepsilon\|_2\le C_\delta.
\end{equation}
The required planar bounds are $\mathcal B[r_\varepsilon-
\langle r_\varepsilon\rangle]\in L^\infty$ since $\beta>2$, and
$r_\varepsilon v_\varepsilon\in L^2(Q_T)$; inertia is controlled by
$\int_0^T\|r_\varepsilon\|_3^2\|v_\varepsilon\|_6^2dt$.
The two diffusion terms are $O_\delta(\sqrt\varepsilon)$.
Time cutoffs of bounded total variation cover $(0,T)$, and
\[
 \varepsilon\nabla r_\varepsilon\to0\text{ in }L^2,
 \qquad\varepsilon(\nabla r_\varepsilon\cdot\nabla)v_\varepsilon
       \to0\text{ in }L^1.
\]

Write $\overline p$ for the weak $L^{(\beta+1)/\beta}$ pressure limit.
The product $p_\delta(r_\varepsilon)r_\varepsilon$ initially has a
locally weak-* measure limit. The full-density effective-flux argument
of \cite[Lemmas~3.2--3.4]{FNP} gives, on interior tests,
\begin{equation}\label{app:innerflux}
 \overline{p_\delta(r)r}-\nu\,\overline{r\operatorname{div}v}
 =\overline p\,r-\nu r\operatorname{div}v,
 \qquad\nu=2\mu+\lambda.
\end{equation}
Its div--curl exponent $2\beta/(\beta+3)\ge4/3$ is compactly embedded
in planar $H^{-1}$; the density potential is locally bounded and
compact because $\beta>2$. Both diffusion errors in the potential
test vanish by \eqref{app:fixedgain} and the $L^2$ boundedness of
$\nabla\Delta^{-1}\operatorname{div}$.

The limiting continuity equation is renormalized since $r\in L^2$
and $\nabla v\in L^2$. Its initial value is retained by the extension
$r=r_0$, $v=0$ to negative times. For the logarithm, the parabolic
inequality is
\[
 \int_\Omega r_\varepsilon\log r_\varepsilon(\tau)
 +\int_0^\tau\!\int r_\varepsilon\operatorname{div}v_\varepsilon
 \le\int_\Omega r_{0,\varepsilon}\log r_{0,\varepsilon}.
\]
Convex terminal lower semicontinuity and the limiting logarithmic
identity therefore give
\[
 \limsup_{\varepsilon\downarrow0}\int_0^\tau\!\int
 r_\varepsilon\operatorname{div}v_\varepsilon
 \le\int_0^\tau\!\int r\operatorname{div}v
\]
at admissible terminal times. Monotonicity makes the pressure-product
defect nonnegative. Increasing interior cutoffs in
\eqref{app:innerflux}, with the uniform
$L^{2\beta/(\beta+2)}$ bound on
$r_\varepsilon\operatorname{div}v_\varepsilon$, and this limsup
inequality yield the Minty identification
$\overline p=p_\delta(r)$; see \cite[Section~3.5]{FNP}.
This exhaustion uses positivity rather than uniform integrability
of the pressure product. Weak pressure convergence tested by one,
together with weak density convergence, gives
$\int r_\varepsilon^\beta\to\int r^\beta$. Uniform convexity then yields
\[
 r_\varepsilon\to r\text{ in }L^\beta(Q_T),\qquad
 p_\delta(r_\varepsilon)\to p_\delta(r)\text{ in }L^1(Q_T).
\]

The weak equations retain $r_0,r_0v_0$. The quadratic renormalization
class follows by truncation, using
\[
 \|r^2\mathbf1_{\{r>R\}}\|_{L^2(Q_T)}^2
 \le R^{4-\beta}\|r\|_{L^\beta(Q_T)}^\beta\to0
\]
for the source paired with $\operatorname{div}v\in L^2$.
For transport, $r^2\mathbf1_{\{r>R\}}\to0$ in
$L^2_tL^{\beta/2}_x$ by time dominated convergence, paired with
$v\in L^2_tL^{\beta/(\beta-2)}_x$. The smooth bounded initial density
makes the initial tails vanish. The logarithm is regularized at zero;
the stated $C^2([0,\infty))$ class is already regular there.

Finally pass the parabolic energy inequality with the weights in
\eqref{app:weightedinner}, discard the nonnegative diffusion term,
and use $\widehat E_{0,\varepsilon}\to\widehat E_\delta(0)$.
Subtracting conserved mass and adding $|\Omega|$ gives
\eqref{app:fixedenergy}. The weak continuous density and momentum
representatives follow from their fixed high-power bounds and weak
equations with initial terms. Nested extractions on the intervals
$(0,L)$, $L=1,2,\ldots$, give the global solution, with all the stated properties.
\end{proof}

\subsection{Natural-energy dependence of the PW estimates}\label{appendix:PW}

\begin{proof}[Proof of Lemma~\ref{app:PW}]
At each fixed $n$, Proposition~\ref{app:fixed} supplies the quadratic
renormalizations and auxiliary integrability in
\cite[Proposition~3.1]{PW}. We trace the estimates in its Sections~4--6.
Write $(r,v)$ for one selected solution and set
\[
 \begin{gathered}
 M=\int_\Omega r_0,\qquad
 K=\operatorname*{ess\,sup}_t\int_\Omega r|v|^2,\qquad
 D=\int_{Q_T}|\nabla v|^2,\\
 P=\operatorname*{ess\,sup}_t\int_\Omega p_\delta(r),\qquad
 J=\|r_0v_0\|_1.
 \end{gathered}
\]
These quantities are uniform in $n$ by \eqref{app:initialenergy} and
\eqref{app:uniformenergy}; in particular $J\le\sqrt{MK_0}$.

\emph{Radon and negative half-derivative bounds.}
In \cite[Section~4]{PW}, the smoothed one-dimensional density
primitive is bounded by $C_\zeta M$. Endpoint momentum terms use
$C_\zeta M(J+\sqrt{MK})$; the inertia and pressure cutoff terms use
$C_\zeta TM(K+P)$, and viscous cutoff terms use
$C_\zeta M\sqrt{TD}$. The principal stress is bounded by
$C_\zeta\sqrt D\sqrt Z$, where $Z$ is the smoothed squared Radon norm.
The principal inertia remainder is nonnegative by the linewise
Cauchy--Schwarz inequality. Consequently,
\[
 Z\le C_{\zeta,T,\Omega}(M,K,D,P,J)+C_\zeta\sqrt D\sqrt Z.
\]
Young's inequality closes this bound uniformly in the line direction
and smoothing scale. Removing smoothing gives \eqref{app:Radon};
\cite[Lemma~2.1 and Proposition~5.1]{PW} then give
\eqref{app:halfderivative}.

\emph{Finite potential exponents.}
On a containing ball let $I_1$ denote the potential with kernel
$|x|^{-1}$, and put
$Q=I_1(\zeta r)$, $L=I_1(\zeta r|v|^2)$.
The half-derivative estimate gives $Q\in L^2_tL^4_x$ locally;
kinetic energy gives $L\in L^\infty_tL^s_x$ for $1<s<2$.
Each component $F$ of the momentum potentials in
\cite[Section~5.1, Lemmas~5.1--5.2]{PW} satisfies $|F|\le C\sqrt{QL}$, hence
$F\in L^4_tL^q_x$ for $q<8/3$. The Sobolev multiplication estimate
\cite[Lemma~2.2]{PW} also gives
$F\in L^1_tL^{4/(1+2\sigma)}_x$ for $0<\sigma<1/2$.
Interpolation with weights $2/3$ and $1/3$, choosing $q$ close to
$8/3$ and $\sigma$ small, therefore yields
\begin{equation}\label{appendix:PWpotential}
 F\in L^2_tL^q_x\quad\text{for every }2<q<3,
\end{equation}
with norms controlled by the same budget and localization.

\emph{Local pressure-density gain.}
Take $\psi(r)=(1+r)^\alpha-1$ in
\cite[Section~6]{PW}. Then
$|\psi(r)|+|r\psi'(r)-\psi(r)|\le C_\alpha r^\alpha$.
Mass and dissipation bound the renormalization source in
$L^2_tL^{2/(1+2\alpha)}_x$. The localized potential
$\mathcal H=\nabla\Delta^{-1}(\zeta\psi(r))$ satisfies
\[
 \|\mathcal H\|_{L^\infty_{t,x}}
 +\|\nabla\mathcal H\|_{L^\infty_tL^{1/\alpha}_x}\le C_{
 \zeta,\alpha,\Omega}(M),
\]
by finite-exponent potential estimates. Choose
$q=2/(1-2\alpha)<3$. The critical potential terms close using
\eqref{appendix:PWpotential} and
\[
 \frac1q+\frac12+\alpha=1.
\]
At $\alpha=1/8$ these exponents are $8/3$, $2$ and $8$,
while the source exponent is $8/5$. Cutoff pressure, inertia and
stress terms are controlled by $P,K,D$ and the preceding potential
bounds. The time endpoint terms satisfy
\[
 |\Gamma_T|\le C_\zeta\|\mathcal H\|_\infty
 \left(\|r_0v_0\|_1+\operatorname*{ess\,sup}_t\|rv(t)\|_1\right)
 \le C_\zeta\|\mathcal H\|_\infty(J+\sqrt{MK}).
\]
Thus no initial velocity Sobolev norm or density lower bound enters
the final constants.

Spatial mollification is removed at each fixed $n$ before asserting
uniformity. At this stage $rv\in L^2(Q_T)$, and the commutator
$[\zeta\psi(r)v]_h-[\zeta\psi(r)]_hv$ tends to zero in $L^2$;
this follows from strong mollifier convergence and time dominated
convergence, using $\psi(r)\in L^\infty_tL^4_x$ and
$v\in L^2_tL^4_x$. The bounded $L^2$ operator
$\nabla\operatorname{div}\Delta^{-1}$ therefore removes the remainder
in \cite[Section~6, (80), (99)--(100), pp.~26, 31--32]{PW}. This fixed-$n$ use of the
momentum norm introduces no outer-uniform high-power requirement.
Since $r^\alpha\le C_\alpha(1+\psi(r))$, the bound for
$\int\zeta^2p_\delta(r)\psi(r)$ and the pressure budget $TP$ give
\eqref{app:localgain}. Its constants
have precisely the dependence stated in Lemma~\ref{app:PW}.
\end{proof}

\section*{Acknowledgments}

The author is partially supported by NSF grant DMS-2510425 and by the Simons
Foundation grant MPS-TSM-00007824.

\section*{AI Use Disclosure}

During the preparation of this work, the author used ChatGPT (OpenAI)  to assist with language editing, checking notation and internal consistency, and LaTeX preparation. The author reviewed and edited any AI-generated output and takes full responsibility for the content of the article.

\end{document}